\documentclass[12pt]{amsart}
\usepackage{amsmath,amsthm,amsfonts,latexsym,amssymb,amscd,color}
\usepackage{graphics,multimedia,tikz}
\newtheorem{thm}{Theorem}[section]
\newtheorem{cor}[thm]{Corollary}

\newtheorem{prp}[thm]{Proposition}
\newtheorem{clm}[thm]{Claim}
\newtheorem*{clm*}{Claim}
\theoremstyle{definition}
\newtheorem{df}[thm]{Definition}

\newtheorem{prb}[thm]{Problem}

\newtheorem{remrk}[thm]{Remark}
\numberwithin{equation}{section}
\newcommand{\sprf}{\noindent{\it Proof.}} 
\newcommand{\sqed}{\hfill\rule{1.3mm}{3mm}\medskip}

\newcommand{\cproof}{\noindent{\it Proof of claim.}\ } 
\newcommand{\cqed}{\hfill\rule{1.3mm}{3mm}}

\newcommand{\wec}[1]{{\mathbf{#1}}}  
\DeclareMathOperator{\Con}{Con}

\DeclareMathOperator{\typ}{typ}

\newcommand{\bd}{\begin{description}}
\newcommand{\ed}{\end{description}}
\newcommand{\var}[1]{{\mathcal{\uppercase{#1}}}}
\newcommand{\class}[1]{{\mathcal{\uppercase{#1}}}}

\newcommand{\al}[1]{\mathbf{#1}}

\newcommand{\ccc}{\mathsf{c}}
\newcommand{\ddd}{\mathsf{d}}
\newcommand{\PP}{\mathsf{P}}

\newcommand{\HH}{\mathbb{H}}
\newcommand{\SSS}{\mathbb{S}}

\newcommand{\diag}{\text{\tt d}}
\newcommand{\frakC}{\frak{C}}

\DeclareMathOperator{\SMP}{\text{\sc SMP}}

\newcommand{\dcoh}{\text{\rm$ \ddd$-coh}}
\newcommand{\dcentr}{\text{\rm $\ddd$-centr}}

\DeclareMathOperator{\SMPd}{\SMP_{\dcoh}}
\DeclareMathOperator{\SMPdc}{\SMP_{\dcentr}}

\newcommand{\us}[1]{^{(#1)}}
\newcommand{\ls}[1]{_{#1}}
\newcommand{\lsus}[2]{_{#1}^{(#2)}}
\newcommand{\ar}{\text{\rm ar}}

\newcommand{\Alg}{\text{\sf Alg}}
\newcommand{\MSAlg}{\text{\sf MSAlg}}
\newcommand{\DAlg}{\text{\sf DAlg}}
\newcommand{\Comma}{\bigl(\Alg(\lngF)\,{\downarrow}\,\al{I}\bigr)}
\newcommand{\comma}{\bigl(\Alg(\lngF)\,{\downdownarrows}\,\al{I}\bigr)}
\newcommand{\Vcomma}{\bigl(\var{V}\,{\downdownarrows}\,\al{I}\bigr)}

\newcommand{\Wcommaell}{\bigl(\var{W}_\ell\,{\downdownarrows}\,\al{I}_\ell\bigr)}
\newcommand{\prodI}{\prod^{\downdownarrows\al{I}}}
\newcommand{\ProdI}{\sideset{}{\kern-2pt{}^{\downdownarrows\al{I}}}\prod}

\newcommand{\lngF}{\mathcal{F}}
\newcommand{\lngG}{\mathcal{G}}
\newcommand{\DvarG}{\mathcal{D}(\hat{\lngG})}
\newcommand{\hatt}[1]{\hat{\text{\tiny(}#1\text{\tiny)}}}

\let\phi=\varphi
\let\epsilon=\varepsilon
\let\bar=\overline
\let\hat=\widehat
\let\tilde=\widetilde

\begin{document}

\title[Algebras from Congruences]{Algebras from Congruences}

\author{Peter Mayr}
\address[Peter Mayr]{Department of Mathematics\\
University of Colorado\\
Boulder CO, USA, 80309-0395}
\email{Peter.Mayr@Colorado.EDU}

\author{\'Agnes Szendrei}
\address[\'Agnes Szendrei]{Department of Mathematics\\
University of Colorado\\
Boulder CO, USA, 80309-0395}
\email{Agnes.Szendrei@Colorado.EDU}

\thanks{This material is based upon work supported by
the Austrian Science Fund (FWF) grant no.\ P24285,
the National Science Foundation grant no.\ DMS 1500254, 
the Hungarian National Foundation for Scientific Research (OTKA)
grants no.\ K104251 and K115518, and
the National Research, Development and Innovation Fund of
Hungary (NKFI) grant no. K128042.
}

\begin{abstract}
 We present a functorial construction which, starting from a congruence $\alpha$ of finite index in an algebra $\al{A}$,
 yields a new algebra $\al{C}$ with the following properties:
 the congruence lattice of $\al{C}$ is isomorphic to the interval of congruences between $0$ and $\alpha$
 on $\al{A}$,
 this isomorphism preserves higher commutators and TCT types,
 and $\al{C}$ inherits all idempotent Maltsev conditions from $\al{A}$.

 As applications of this construction, we first show that supernilpotence is decidable for congruences of finite
 algebras in varieties that omit type $\mathbf{1}$.
 Secondly, we prove that 
 the subpower membership problem for finite algebras with a cube term can be effectively reduced to membership
 questions in subdirect products of subdirectly irreducible algebras with central monoliths.
 As a consequence, we obtain a polynomial time algorithm for the subpower membership problem for finite algebras with a cube
 term in which the monolith of every subdirectly irreducible section has a supernilpotent centralizer.  
\end{abstract}

\maketitle

\section{Introduction}
\label{sec-intro}
This paper was motivated by Problems~\ref{prb-1}--\ref{prb-2} below.
In each problem 
the question is whether or not a result that is known for certain algebras
can be lifted to congruences.

\begin{prb}
\label{prb-1}
Let $\al{A}$ be a finite algebra (in a finite language). 
Given a congruence $\alpha$ of $\al{A}$, is it decidable if $\alpha$ is 
supernilpotent?
\end{prb}

This question arises from the result that finite nilpotent Maltsev algebras are non-dualizable if
they have some non-abelian supernilpotent congruence~\cite{BM:SPD}.

In the special case when
$\alpha$ is the total congruence of $\al{A}$,
i.e., when the question is whether it is decidable if $\al{A}$ itself
is supernilpotent, the answer to Problem~\ref{prb-1}
has been known to be YES by a combination of
results from \cite{hobby-mckenzie, Ke:CMVS, AM:SAHC}
(see Section~\ref{sec-appl1}), provided 
the variety generated by $\al{A}$ is assumed to omit type $\mathbf{1}$.

\begin{prb}
\label{prb-2}
Let $\al{A}$ be a finite algebra (in a finite language) such that $\al{A}$
has a cube term. 
If for every subdirectly irreducible section 
$\al{S}$ of $\al{A}$ the centralizer of the monolith of 
$\al{S}$ is supernilpotent, 
does there exist a polynomial time algorithm for solving the 
Subpower Membership Problem for $\al{A}$?
\end{prb}

In the special case when $\al{A}$ itself is supernilpotent 
(and, without loss of generality, has prime power order), 
the answer has been known to be YES by a result of 
\cite{Ma:SMP}.

Our goal in this paper is to use the composition of two known functors to 
create ``algebras from congruences'', to study 
the resulting (functorial) construction, and show that it preserves
many algebraic properties.  Then we apply these results to prove that
the answer to Problem~\ref{prb-2} is YES, and 
the answer to Problem~\ref{prb-1} is also YES provided 
the variety generated by $\al{A}$ omits type $\mathbf{1}$.
For finite algebras $\al{A}$ not satisfying this assumption,
Problem~\ref{prb-1} remains open even for the special case when
$\alpha$ is the total congruence of $\al{A}$.

To describe the idea of the construction we use, let us look at a
single algebra $\al{A}$ and a congruence $\alpha$ of $\al{A}$
with finitely many blocks. The first step of the construction 
yields a multisorted algebra
in which the sorts are the $\alpha$-blocks, and the multisorted
operations are the restrictions of the operations of $\al{A}$ to the sorts
in all possible ways.
The second step of the construction
takes this multisorted algebra as its input, and yields a single-sorted algebra
on the product of the sorts, where the operations are the diagonal operation
of the product of the sorts and some totally defined operations that faithfully
represent the operations of the multisorted algebra.

By restricting to pairs $(\al{A},\alpha)$ where
$\alpha$ is the kernel of a homomorphism $\al{A}\to\al{I}$ of $\al{A}$
into a fixed finite algebra $\al{I}$, the first step of the construction 
can be made into a functor $\frak{M}$  
from a category of algebras over $\al{I}$ (cf.~\cite[Ch.~II, Sec.~6]{maclane})
into a category of multisorted algebras, while the second step of the construction
can be made into a functor $\frak{P}$ from a category of
multisorted algebras to a category of (single-sorted) algebras.
We will introduce these functors in Section~\ref{sec-constr}, and will
prove that their composition $\frak{C}:=\frak{P}\circ\frak{M}$ ---
which is our construction of ``algebras from congruences'' ---
is a categorical equivalence, provided we restrict to the cases
where the homomorphism $\al{A}\to\al{I}$ is onto, that is, the sorts
of the corresponding multisorted algebras are nonempty.

The functors $\frak{M}$ and $\frak{P}$ were introduced in the 1980's by 
Novotn\'y~\cite{novotny} and Gardner~\cite{gardner}, respectively.
A precursors of $\frak{P}$ appears in \cite{barr} (see also the last section
of \cite{gog-mes}).
Freese and McKenzie~\cite{freese-mckenzie} used 
a variant of the construction provided by the functor
$\frak{C}$
to associate modules (over appropriate rings) to 
abelian congruences of algebras in a congruence modular variety,
which has played a prominent role in commutator theory for over three 
decades.
The constructions afforded by the functors $\frak{M}$ and $\frak{P}$
were also applied by McKenzie and Valeriote~\cite{mckenzie-valeriote}
to uncover the structure of decidable, strongly abelian, locally finite
varieties.
In the 1990's, VanderWerf~\cite{VanderWerf} 
used a functor similar to $\frak{M}$ for his
construction of a ``derived algebra'' (a multi-sorted algebra), a
functor which is essentially the same as $\frak{P}$ for his construction
called ``consolidation'', and their composition $\frak{C}$, 
to generalize the Krohn--Rhodes wreath
decomposition theory of finite transformation semigroups to arbitrary
finite algebras.
More recently, the functor $\frak{P}$ was applied 
by Mu\v{c}ka, Romanowska, and Smith~\cite{mucka-r-s} 
to find good sets of explicit defining identities and quasi-identities 
for the variety of single-sorted algebras equivalent to
the class of all multisorted algebras of a given language 
where either all sorts are nonempty or all sorts are empty.

In most applications of the categorical equivalence $\frak{C}$, including
those mentioned in the preceding paragraph, a central role is played by the
algebraic properties preserved by $\frak{C}$. 
To answer the questions in Problems~\ref{prb-1}--\ref{prb-2} we will use
a wide array of these properties. Since many of these results have not
been published in the literature in the generality we need them, 
we devote Section~3 to a survey of the algebraic properties of $\frak{C}$. 
To describe the properties we consider,
it will be convenient to introduce a notation for the $\frak{C}$-image
of an algebra $\al{A}$ (over $\al{I}$). 
In this section, we will use the informal notation
$\al{A}^{\frak{C}}$. (The precise definition of $\frak{C}$ and the notation that 
goes along with it, and is used outside this section, 
can be found in Corollary~\ref{cor-functorC}.)

For congruences,
we prove in Section~\ref{sec-functorC} 
that for every algebra $\al{A}$ and surjective homomorphism
$\al{A}\to\al{I}$ with kernel $\alpha$, 
the functor $\frak{C}$ yields an isomorphism 
$I_{\al{A}}(0,\alpha)\to\Con(\al{A^{\frak{C}}})$
between the interval
$I_{\al{A}}(0,\alpha) = \{ \beta\in\Con(\al{A}) : 0\leq\beta\leq\alpha \}$ of the congruence lattice of $\al{A}$
and the whole congruence lattice $\Con(\al{A}^{\frak{C}})$ 
of $\al{A}^{\frak{C}}$ (see Corollary~\ref{cor-congr}).
Moreover, we show that this isomorphism 
preserves commutators and higher arity commutators
(see Theorem~\ref{thm-comm}). 
In particular,
$\alpha$ is a supernilpotent congruence of $\al{A}$ iff 
$\al{A}^{\frak{C}}$ is a supernilpotent algebra, 
and the degrees of supernilpotence are the same.
In addition, we establish that the functor
$\frak{C}$ preserves common finiteness conditions. For example,
$\al{A}$ is finite, finitely generated,
finitely presented, or
residually finite, respectively, iff $\al{A}^{\frak{C}}$ is
(see
  Theorem~\ref{thm-subprod}(4) and
  Corollaries~\ref{cor-presentations}, \ref{cor-quotient}(3)).
We also clarify the relationship between the clones of 
$\al{A}$ and $\al{A}^{\frak{C}}$, by explicitly 
describing how one can obtain the clone of term
operations and the clone of polynomial operations of $\al{A}^{\frak{C}}$ 
from the corresponding clone of $\al{A}$
(see Corollaries~\ref{cor-termops}(1) and \ref{cor-poly}). 
An important consequence of this relationship 
between the clones of term operations of $\al{A}$ and $\al{A}^{\frak{C}}$
is that 
the variety generated by $\al{A}^{\frak{C}}$ satisfies all idempotent
Maltsev conditions that hold in the variety generated by $\al{A}$
(cf.~Corollary~\ref{cor-clone-idem}).
Finally, we use the relationship between the clones of polynomial operations
of $\al{A}$ and $\al{A}^{\frak{C}}$ to 
prove that for finite $\al{A}$,  
the tame congruence theoretic types of covering pairs are preserved
by the isomorphism 
$I_{\al{A}}(0,\alpha)\to\Con(\al{A^{\frak{C}}})$ mentioned above
(see Theorem~\ref{thm-tct}).
Variants or particular cases of these results
appear in the
applications of $\frak{C}$ in 
\cite{freese-mckenzie, mckenzie-valeriote, VanderWerf}
mentioned above (see, e.g.,
\cite[Theorem~9.9]{freese-mckenzie},
\cite[Chapter~11]{mckenzie-valeriote}, 
\cite[Lemmas~2.24, 5.2]{VanderWerf}).

Sections~\ref{sec-appl1} and \ref{sec-appl2} answer Problem~\ref{prb-1}
(for varieties omitting type {\bf 1})
and Problem~\ref{prb-2}, respectively, in the affirmative.
In Section~\ref{sec-appl1} (Theorem~\ref{thm-sn}),
for finite algebras in a variety omitting type {\bf 1},
we give a characterization
for a congruence $\alpha$ to be supernilpotent,
which
immediately proves that supernilpotence is decidable.
Recently, essentially
the same characterization of supernilpotent congruences (in congruence modular varieties) has
also been observed
and used in algorithms for circuit satisfiability problems by Idziak,
 Kawa{\l}ek and Kraczkowski~\cite{IKK}.
In Section~\ref{sec-appl2} the theorems we prove apply to
finite sets of finite algebras --- not just single finite algebras.
First we show that 
the subpower membership problem for a finite set of finite 
algebras with a cube term is polynomial time reducible to its subproblem 
where the algebras are subdirectly irreducible and have central monoliths
(see Theorem~\ref{thm-dcentr-reduc}(3)).
From this we derive a common generalization of the main results of 
\cite{BMS:SMP} and \cite{Ma:SMP} (see Theorem~\ref{thm-P-snilp-centrs}),
which answers Problem~\ref{prb-2}.

\section{Three equivalent categories}
\label{sec-constr}

Throughout this paper $\lngF$ will denote an algebraic language.
So, for each symbol $f\in\lngF$ there is an associated natural number,
$\ar(f)\ge0$, the \emph{arity} of $f$, which indicates that in every
$\lngF$-algebra $\al{A}=(A;\lngF)$ each symbol $f$ is interpreted as
an $\ar(f)$-ary operation $f^{\al{A}}\colon A^{\ar(f)}\to A$.
To simplify notation, we will usually drop the superscript $\al{A}$
from the operations $f^{\al{A}}$.
If $\lngF$ contains no nullary symbols, we will allow the universe
$A$ of an $\lngF$-algebra to be empty.

Analogously, $\lngG$ will always denote a multisorted algebraic language.
There is an associated nonempty set, $J$, indexing the sorts,
and for each symbol $g\in\lngG$ there is an associated natural number
$\ar(g)\ge0$, the \emph{arity} of $g$, and an associated pair
$(\mathbf{j},j')$, where $\al{j}=(j_1,\dots,j_{\ar(g)})\in J^{\ar(g)}$ is the
\emph{sequence of input sorts} of $g$ and $j'\in J$ is
the \emph{output sort} of $g$.
These data indicate that in every
multisorted $\lngG$-algebra $\al{M}=\bigl((M^{(j)})_{j\in J};\lngG\bigr)$
each symbol $g$ is interpreted as
a multisorted operation
$g^{\al{M}}\colon M^{(j_1)}\times\dots\times M^{(j_{\ar(g)})}\to M^{(j')}$.
To simplify notation, we will usually drop the superscript $\al{M}$
from the multisorted operations $g^{\al{M}}$.
For each $j'\in J$, if $\lngG$ contains no nullary symbol with
output sort $j'$, 
we will allow the universe $M^{(j')}$ of a multisorted
$\lngG$-algebra to be empty.

Next we introduce several categories.

\begin{df}
  \label{df-cat1}
  \begin{enumerate}
  \item[(1)]
  $\Alg(\lngF)$ will denote the category of all $\lngF$-algebras; that is,
  the objects are all $\lngF$-algebras, and the
  morphisms are all homomorphisms between $\lngF$-algebras.
  \item[(2)]
  Given a nonempty $\lngF$-algebra $\al{I}=(I;\mathcal{F})$, 
  the category
  $\Comma$ of all $\lngF$-algebras over $\al{I}$ is defined
  (cf.~\cite[Ch.~II, Sec.~6]{maclane}) to be the category in which
  \begin{itemize}
  \item
    the objects are all pairs $(\al{A},\chi)$ where $\al{A}$ is an
    $\lngF$-algebra and $\chi$ is a homomorphism
    $\al{A}\to\al{I}$ 
    of $\lngF$-algebras, while
  \item
    a morphism between two such objects $(\al{A},\chi)$ and $(\al{B},\xi)$
    is a homomorphism $\psi\colon\al{A}\to\al{B}$ of $\lngF$-algebras
    such that $\chi=\xi\circ \psi$.
  \end{itemize}
  \item[(3)]
  $\MSAlg(\lngG)$ will denote the category of all multisorted
  $\lngG$-algebras; that is,
  the objects are all multisorted $\lngG$-algebras, and the
  morphisms are all homomorphisms between multisorted $\lngG$-algebras.
  Recall that a homomorphism $\zeta\colon \al{M}\to\al{N}$
  between multisorted $\lngG$-algebras
  $\al{M}=\bigl((M^{(j)})_{j\in J};\lngG\bigr)$  and
  $\al{N}=\bigl((N^{(j)})_{j\in J};\lngG\bigr)$
  is a $J$-tuple
  $\zeta=(\zeta^{(j)})_{j\in J}$ of functions $\zeta^{(j)}\colon M^{(j)}\to N^{(j)}$
  such that $\zeta$ preserves all multisorted operations $g\in\lngG$.
  \end{enumerate}
\end{df}

The next statement introduces a functor 
$\frak{M}\colon\Comma\to\MSAlg(\lngG)$ for an appropriately defined
multisorted language $\lngG$, which is analogous to a functor considered
in \cite{novotny}.

\begin{prp}
  \label{prp-functorM}
  Given an algebraic language $\mathcal{F}$ and a nonempty $\mathcal{F}$-algebra
  $\al{I}=(I,\mathcal{F})$, let $\lngF_{\al{I}}$ be
  the multisorted language $\lngF_{\al{I}}$ defined as follows:
  \begin{itemize}
  \item
    the sorts are indexed by the set $I$, and
  \item
    the operation symbols of the multisorted language
    $\mathcal{F}_{\al{I}}$ are all symbols
    \begin{multline*}
      \qquad\quad
      f_{\wec{i}}\ \
      \text{with } f\in\mathcal{F} \text{ and } \wec{i}\in I^{\ar(f)},\\
      \text{\ \ where $\ar(f_{\wec{i}})=\ar(f)$,\ \ 
        $\wec{i}$ is the sequence of input sorts,}\\
      \text{and $f(\wec{i})$ (computed in $\al{I}$) is the output sort.}
      \quad
    \end{multline*}
  \end{itemize}
  The following functions define a functor
  $\frak{M}\colon\Comma\to\MSAlg(\lngF_{\al{I}})$:
  \begin{itemize}
  \item
  $\frak{M}$ assigns to every
  object $(\al{A},\chi)$ of $\Comma$ the multisorted algebra
  \[
  \frak{M}(\al{A},\chi):=\bigl((\chi^{-1}(i))_{i\in I};\lngF_{\al{I}}\bigr)
  \]
  where the sorts are the congruence classes of the kernel of $\chi$
  (indexed by $\chi$), and for each
  $f\in\lngF$ and for every tuple
  $\wec{i}=(i_1,\dots,i_{\ar(f)})\in I^{\ar(f)}$, the operation $f_{\wec{i}}$ is the
  restriction of the operation $f$ of $\al{A}$ to the set
  $\chi^{-1}(i_1)\times\dots\times\chi^{-1}(i_{\ar(f)})$.
  \item
  $\frak{M}$ assigns to every morphism $\psi\colon (\al{A},\chi)\to(\al{B},\xi)$
  of $\Comma$ the multisorted homomorphism
  \[
  \frak{M}(\psi):=\bigl(\psi\us{i}\bigr)_{i\in I}\colon
  \frak{M}(\al{A},\chi)\to\frak{M}(\al{B},\xi)
  \]
  where for each $i\in I$, $\psi\us{i}$ is the restriction of $\psi$ to
  $\chi^{-1}(i)$, which is a function $\chi^{-1}(i)\to\xi^{-1}(i)$.
  \end{itemize}
\end{prp}

\begin{df}
  \label{df-cat2}
  \begin{enumerate}
  \item[(1)]
  Given a nonempty $\lngF$-algebra $\al{I}=(I;\mathcal{F})$,
  $\comma$ will be our notation for 
  the full subcategory of $\Comma$,
  which consists of all objects $(\al{A},\chi)$ such that $\chi$ is onto.
  \item[(2)]
  $\MSAlg^\Box(\lngG)$ will denote
  the full subcategory of $\MSAlg(\lngG)$,
  which consists of all multisorted $\lngG$-algebras
  $\al{M}=\bigl((M^{(j)})_{j\in J};\lngG\bigr)$ whose sorts
  $M^{(j)}$ ($j\in J$) are pairwise disjoint.
  \item[(3)]
  $\MSAlg^+(\lngG)$ will stand for
  the full subcategory of $\MSAlg(\lngG)$,
  which consists of all multisorted $\lngG$-algebras
  $\al{M}=\bigl((M^{(j)})_{j\in J};\lngG\bigr)$ whose sorts
  $M^{(j)}$ ($j\in J$) are all nonempty.
  \item[(4)]
    Finally, $\MSAlg^\boxplus(\lngG)$ will be our notation for the full
    subcategory of $\MSAlg(\lngG)$ whose objects are the multisorted
    $\lngG$-algebras that belong to both
    $\MSAlg^\Box(\lngG)$ and $\MSAlg^+(\lngG)$.
  \end{enumerate}
\end{df}

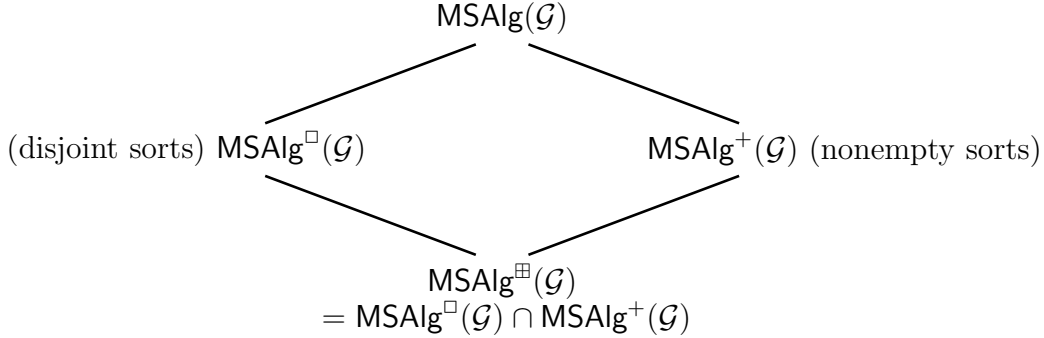
\setlength{\unitlength}{1truemm}
\begin{picture}(100,45)
\put(5,0){%
\begin{tikzpicture}[scale=.35]
  \node at (0,10) {$\MSAlg(\lngG)$};
  \node at (-12,5) {(disjoint sorts) $\MSAlg^\Box(\lngG)$};
  \node at (13,5) {$\MSAlg^+(\lngG)$ (nonempty sorts)};
  \node at (0,0) {$\MSAlg^\boxplus(\lngG)$};
  \node at (0,-1.4) {${}=\MSAlg^\Box(\lngG)\cap\MSAlg^+(\lngG)$};
  \draw[line width = 1pt] (-1,9) -- (-9,6);
  \draw[line width = 1pt] (1,9) -- (9,6);
  \draw[line width = 1pt] (1,1) -- (9,4);
  \draw[line width = 1pt] (-1,1) -- (-9,4);
\end{tikzpicture}
}
\end{picture}

\begin{thm}
  \label{thm-functorM}
    For any algebraic language $\mathcal{F}$ and any nonempty
    $\mathcal{F}$-algebra
  $\al{I}=(I,\mathcal{F})$, the category $\Comma$ of $\lngF$-algebras over 
  $\al{I}$ is equivalent to
  the category $\MSAlg(\lngF_{\al{I}})$ of multisorted $\lngF_{\al{I}}$-algebras.
  The full subcategories $\comma$ and\break $\MSAlg^+(\lngF_{\al{I}})$
  of these categories are also equivalent.

  In more detail, we have the following.
  \begin{enumerate}
  \item[{\rm(1)}]
    The functor $\frak{M}$ described in Proposition~\ref{prp-functorM}
    maps onto the full subcategory
    $\MSAlg^\Box(\lngF_{\al{I}})$ of $\MSAlg(\lngF_{\al{I}})$, and yields
    (by changing the target category)
    an isomorphism
    \[
    \frak{M}\colon\Comma\to\MSAlg^\Box(\lngF_{\al{I}}).
    \]
   \item[{\rm(2)}]
    $\frak{M}$ restricts to an isomorphism between the
    full subcategory
    $\comma$ of $\Comma$ and the full subcategory 
    $\MSAlg^\boxplus(\lngF_{\al{I}})$
    of $\MSAlg^\Box(\lngF_{\al{I}})$.
  \end{enumerate}
\end{thm}  

\begin{proof}
First we prove statement (1) by showing that the functor
\[
\frak{M}^{-1}\colon\MSAlg^\Box(\lngF_{\al{I}})\to\Comma
\]
defined below is an inverse of $\frak{M}$:
    \begin{itemize}
  \item
  $\frak{M}^{-1}$ assigns to every
    multisorted $\lngF_{\al{I}}$-algebra
    $\al{M}=\bigl((M\us{i})_{i\in I};\lngF_{\al{I}}\bigr)$
    in $\MSAlg^\Box(\lngF_{\al{I}})$ the pair 
    \[
    \frak{M}^{-1}(\al{M}):=
    \left(\Bigl(\bigcup_{i\in I}M\us{i};\lngF\Bigr),\chi\right)
    \]
    where for each symbol $f\in\lngF$, the operation $f$ of the
    $\lngF$-algebra $\al{A}=\bigl(\bigcup_{i\in I}M\us{i};\lngF\bigr)$ is
    the union of the multisorted operations
    $f_{\al{i}}$ for all $\al{i}\in I^{\ar(f)}$, and
    the homomorphism $\chi\colon\al{A}\to\al{I}$ sends each
    element $a\in\bigcup_{i\in I}M\us{i}$ to the unique $i\in I$ with
    $a\in M\us{i}$.
  \item
    $\frak{M}^{-1}$ assigns to every homomorphism
    $\zeta=(\zeta\us{i})_{i\in I}\colon\al{M}\to\al{N}$
    the morphism 
    \[
    \frak{M}^{-1}(\zeta):=\bigcup_{i\in I}\zeta\us{i}\colon
    \frak{M}^{-1}(\al{M}) \to \frak{M}^{-1}(\al{N})
    \]
    in $\Comma$.
  \end{itemize}
Let $\al{M}=\bigl((M\us{i})_{i\in I};\lngF_{\al{I}}\bigr)$ be a
multisorted $\lngF_{\al{I}}$-algebra in $\MSAlg^\Box(\lngF_{\al{I}})$, and let
$J:=\{j\in J:M\us{j}\not=\emptyset\}$.
Notice first that $J$ is a subuniverse of $\al{I}$; indeed, if
$f\in\lngF$ and $\al{i}=(i_1,\dots,i_{\ar(f)})\in J^{\ar(f)}$, then
$\al{M}$ has an operation $f_{\al{i}}$ with nonempty domain 
$M\us{i_1}\times\dots\times M\us{i_{\ar(f)}}$
and with codomain $M\us{f(\al{i})}$,
so $M\us{f(\al{i})}$ must be nonempty, i.e., $f(\al{i})\in J$.
It follows also that for every tuple 
$\al{i}\in I^{\ar{f}}\setminus J^{\ar{f}}$, the multisorted operation
$f_{\al{i}}$ of $\al{M}$ is the empty function.
Thus, the algebra $\al{A}=\bigl(\bigcup_{i\in I}M\us{i};\lngF\bigr)$ in
$\frak{M}^{-1}(\al{M})$ described above is indeed an $\lngF$-algebra,
and $\chi\colon\al{A}\to\al{I}$ is indeed an $\lngF$-algebra homomorphism.  
Thus, the object function of $\frak{M}^{-1}$ is well-defined.

To check the analogous statement for the morphisms,
let $\zeta=(\zeta\us{i})_{i\in I}\colon\al{M}\to\al{N}$ be a homomorphism in
$\MSAlg(\lngF_{\al{I}})$, and let $\frak{M}^{-1}(\al{M})=(\al{A},\chi)$
and $\frak{M}^{-1}(\al{N})=(\al{B},\xi)$. 
It is clear from the definitions
of $\al{A}$ and $\al{B}$ that $\bigcup_{i\in I}\zeta\us{i}$ is a homomorphism
$\al{A}\to\al{B}$. We also have $\xi\circ\bigcup_{i\in I}\zeta\us{i}=\chi$,
for the following reason:
for every $a\in A$ there is a unique $i\in I$ with $a\in M\us{i}$, so
$\chi(a)=i$ and $M\us{i}\not=\emptyset$; 
since $\zeta\us{i}\colon M\us{i}\to N\us{i}$, 
we also have $N\us{i}\not=\emptyset$,
whence $(\xi\circ\bigcup_{i\in I}\zeta\us{i})(a)=\xi(\zeta\us{i}(a))=i$.
This shows that $\frak{M}^{-1}(\zeta)=\bigcup_{i\in I}\zeta\us{i}$ is indeed 
a morphism $\frak{M}^{-1}(\al{M})\to\frak{M}^{-1}(\al{N})$ in $\Comma$.  

To complete the proof of statement~(1), one can easily verify
from the definitions that $\frak{M}^{-1}$ is a functor such that 
$\frak{M}\circ\frak{M}^{-1}$ is the identity functor on $\MSAlg(\lngF_{\al{I}})$
and
$\frak{M}^{-1}\circ\frak{M}$ is the identity functor on $\Comma$.

Statement (2) follows immediately from statement (1).

Finally,
the first statement of the theorem about the equivalence of the categories
$\Comma$ and $\MSAlg(\lngF_{\al{I}})$ follows from the isomorphism
of $\Comma$ and $\MSAlg^\Box(\lngF_{\al{I}})$ established in (1), and the fact
that every multisorted algebra is isomorphic to one in which the sorts are
pairwise disjoint.
\end{proof}

Next we discuss a variant of the main result of
\cite{gardner}.

\begin{prp}
  \label{prp-functorP}
  Given a multisorted algebraic language $\lngG$ with no nullary symbols
  and finitely many sorts
  indexed by the set $[m]:=\{1,\dots,m\}$ $(m>0)$, let $\hat{\lngG}$
  be the algebraic language whose symbols are
  \begin{itemize}
    \item
      $\diag$ with  $\ar(\diag)=m$, and
    \item
      $\hat{g}$ with $\ar(\hat{g})=\ar(g)$ for all $g\in G$.
  \end{itemize}
  The following functions define a functor
  $\frak{P}\colon\MSAlg(\lngG)\to\Alg(\hat{\lngG})$:
    \begin{itemize}
    \item
      $\frak{P}$ assigns to every multisorted $\lngG$-algebra
      $\al{M}=\bigl((M\us{i})_{i\in[m]};\lngG\bigr)$ the $\hat{\lngG}$-algebra
      \[
      \frak{P}(\al{M}):=(M\us{1}\times\dots\times M\us{m};\hat{\lngG})
      \]
      where
      \begin{itemize}
      \item
        $\diag$ is the \emph{diagonal operation} on the product set
      $M\us{1}\times\dots\times M\us{m}$; that is, for any tuples
      $(a\lsus{j}{i})_{i\in[m]}\in \prod_{i\in[m]}M\us{i}$
      $(j\in[m])$,
      \begin{equation}
      \label{eq-constrd}
      \diag\bigl((a\lsus{1}{i})_{i\in[m]},\dots,(a\lsus{m}{i})_{i\in[m]}\bigr)
      =(a\lsus{1}{1},\dots,a\lsus{m}{m}),
      \end{equation}
      and
    \item
      for every $g\in\lngG$, if the input and output sorts are
      $\wec{i}=(i_1,\dots,i_k)\in[m]^k$ and $i'\in[m]$ (so, $\ar(g)=k$),
      then the $k$-ary operation $\hat{g}$ is the following:
      for any input tuples
      $(a\lsus{j}{i})_{i\in[m]}\in \prod_{i\in[m]}M\us{i}$
      $(j\in[k])$, 
    \begin{multline}
      \label{eq-constrg}
      \qquad\quad
      \hat{g}
      \bigl((a\lsus{1}{i})_{i\in[m]},\dots,(a\lsus{k}{i})_{i\in[m]}\bigr)\\
        =\bigl(a\lsus{1}{1},\dots,a\lsus{1}{i'-1},
          g(a\lsus{1}{i_1},\dots,a\lsus{k}{i_k}),
          a\lsus{1}{i'+1},\dots,a\lsus{1}{m}\bigl).
          \qquad\quad
    \end{multline}
    \end{itemize}
  \item
    $\frak{P}$ assigns to every morphism
    $\zeta=(\zeta\us{i})_{i\in[m]}\colon\al{M}\to\al{N}$ in
    $\MSAlg(\lngG)$,
    where each $\zeta\us{i}$ $(i\in[m])$ is a function $M\us{i}\to N\us{i}$,
    the morphism
    \[
    \frak{P}(\zeta):=\zeta\us{1}\times\dots\times\zeta\us{m}\colon
    \frak{P}(\al{M})\to\frak{P}(\al{N})
    \]
    in $\Alg(\hat{\lngG})$,
    which is a function
    $M\us{1}\times\dots\times M\us{m}\to N\us{1}\times\dots\times N\us{m}$.
  \end{itemize}
\end{prp}

As it is remarked in \cite{gardner}, for a fixed $\lngG$,
the class of all isomorphic copies of algebras of the form $\frak{P}(\al{M})$,
as described in Proposition~\ref{prp-functorP},
is a variety $\DvarG$ (including the empty algebra);
$\DvarG$ is defined by the following identities:
\begin{align}
  \diag(x,\dots,x) & {}=x,\label{eq-diag1}\\
  \diag\bigl(\diag(x_{11},\dots,x_{1m}),\diag(x_{21},\dots,x_{2m}),
  \dots,\diag(x_{m1},\dots,x_{mm})\bigr)
   & {}=\diag(x_{11},\dots,x_{mm}),\label{eq-diag2}
\end{align}  
and for every symbol $g\in\lngG$ with input and output sorts
$\al{i}=(i_1,\dots,i_k)\in[m]^k$ and $i'\in[m]$ (so, $\ar(g)=\ar(\hat{g})=k$),
\begin{align}
  \diag_{i'}\bigl(\hat{g}(x_1,\dots,x_k),y\bigr)
& {}=\diag_{i'}\bigl(\hat{g}\bigl(\diag_{i_1}(x_1,z_1),\dots,\diag_{i_k}(x_k,z_k)\bigr),y\bigr),
  \label{eq-g1}\\
  \diag_j\bigl(\hat{g}(x_1,\dots,x_k),y\bigr)
  & {}=\diag_j(x_1,y)\quad\text{for all $j\in[m]\setminus\{i'\}$,\label{eq-g2}}
\end{align}
where $\diag_\ell(x,y)$ is an abbreviation for $\diag(y,\dots,y,x,y,\dots,y)$
with $x$ in the $\ell$-th position.

Briefly, the statement that every algebra $\al{A}=(A;\hat{\lngG})$
in $\DvarG$
is isomorphic to an algebra of the form $\frak{P}(\al{M})$
can be proved as follows.
The identities \eqref{eq-diag1}--\eqref{eq-diag2} imply that
for each $\ell\in[m]$ there exists an equivalence relation $\equiv_\ell$
on $A$ such that for any $a,b\in A$ we have
\begin{equation*}
  a\equiv_\ell b\quad \text{iff}\quad
  a=\diag_\ell(b,a)\quad \text{iff}\quad
  \diag_\ell(a,c)=\diag_\ell(b,c)\ \text{for all $ c\in A$.}
\end{equation*}
Moreover, ${\equiv}_1,\dots,{\equiv}_m$ yield a product decomposition
$A\to A/{\equiv}_1\times\dots\times A/{\equiv_m}$ so that
for each $\ell\in[m]$ and $a\in A$ the elements of $A$ with the same
$\ell$-th coordinate as $a$ are exactly the elements of the form
$\diag_\ell(a,c)$ ($c\in A$).
Thus, the identities \eqref{eq-g1}--\eqref{eq-g2} express that
for any elements $a_1,\dots,a_k\in A$, the $i'$-th coordinate of
$\hat{g}(a_1,\dots,a_k)$ depends only on the coordinates $i_1,\dots,i_k$
of $a_1,\dots,a_k$, respectively, while for $j\in[m]\setminus\{i'\}$,
the $j$-th coordinate of $\hat{g}(a_1,\dots,a_k)$ is the $j$-th
coordinate of $a_1$.

\begin{df}\label{df-cat3}
  \begin{enumerate}
    \item[(1)]
  $\DAlg(\hat{\lngG})$ will denote the full subcategory of $\Alg(\hat{\lngG})$
  whose objects are the $\hat{\lngG}$-algebras of the form $\frak{P}(\al{M})$
  for some multisorted $\lngG$-algebra $\al{M}$.
  \item[(2)]
  We will use the notation
  $\DAlg^+(\hat{\lngG})$ for the full subcategory of $\DAlg(\hat{\lngG})$
  obtained by omitting the empty algebra; that is, the objects of
  $\DAlg^+(\hat{\lngG})$ are exactly the algebras $\frak{P}(\al{M})$
  for which the sorts of $\al{M}$ are all nonempty.
  \item[(3)]
  $\DAlg^\boxplus(\hat{\lngG})$ will be our notation for the
  full subcategory of $\DAlg^+(\hat{\lngG})$
  consisting of those algebras $\frak{P}(\al{M})$
  where the sorts of $\al{M}$ are pairwise disjoint.
  \item[(4)]
  $\mathcal{D}^+(\hat{\lngG})$ will denote the full subcategory of
  $\Alg(\hat{\lngG})$ whose objects are the nonempty members of the variety
  $\DvarG$.
  \end{enumerate}  
\end{df}

\setlength{\unitlength}{1truemm}
\begin{picture}(100,52)
\put(25,0){%
\begin{tikzpicture}[scale=.30]
  \node at (0,10) {$\DvarG$};
  \node at (-10,5) {$\DAlg(\hat{\lngG})$};
  \node at (15,5) {$\mathcal{D}^+(\hat{\lngG})
    =\DvarG\setminus\{(\emptyset;\hat{\lngG})\}$};
  \node at (6.1,0) {$\DAlg^+(\hat{\lngG})
    =\DAlg(\hat{\lngG})\cap\mathcal{D}^+(\hat{\lngG})$};
  \node at (0,-5) {$\DAlg^\boxplus(\hat{\lngG})$};
  \draw[line width = 1pt] (-1,9) -- (-9,6);
  \draw[line width = 1pt] (1,9) -- (9,6);
  \draw[line width = 1pt] (1,1) -- (9,4);
  \draw[line width = 1pt] (-1,1) -- (-9,4);
  \draw[line width = 1pt] (0,-1) -- (0,-4);
\end{tikzpicture}
}
\end{picture}

Clearly,
$\DAlg^+(\hat{\lngG})$ is a full subcategory of
$\mathcal{D}^+(\hat{\lngG})$.

\begin{thm}
  \label{thm-functorP}
For any multisorted algebraic language $\lngG$ with no nullary symbols and
finitely many sorts, the category $\MSAlg^+(\lngG)$ of multisorted
$\lngG$-algebras is equivalent to the category
$\mathcal{D}^+(\hat{\lngG})$ of $\hat{\lngG}$-algebras.

  In more detail, we have the following.
  \begin{enumerate}
  \item[{\rm(1)}]
    The functor $\frak{P}$ described in Proposition~\ref{prp-functorP}
    maps the full subcategory $\MSAlg^+(\lngG)$ of $\MSAlg(\lngG)$ into
    the full subcategory $\DAlg^+(\hat{\lngG})$ of
    $\Alg(\hat{\lngG})$, and yields
    (by changing the domain and target categories)
    an isomorphism
    \[
    \frak{P}\colon\MSAlg^+(\lngG)\to\DAlg^+(\hat{\lngG}).
    \]
   \item[{\rm(2)}]
    $\frak{P}$ restricts to an isomorphism between the
    full subcategory
    $\MSAlg^\boxplus(\lngG)$ of $\MSAlg^+(\lngG)$ and the full subcategory 
    $\DAlg^\boxplus(\hat{\lngG})$
    of $\DAlg^+(\hat{\lngG})$.
  \end{enumerate}
\end{thm}  

\begin{proof}
As in the proof of Theorem~\ref{thm-functorM}, the crucial statement is (1), 
because (2) is an immediate consequence of (1), while the first 
statement of the theorem on the equivalence of the categories 
$\MSAlg^+(\lngG)$ 
and 
$\mathcal{D}^+(\hat{\lngG})$
follows from (1) and the fact that 
every algebra in 
$\mathcal{D}^+(\hat{\lngG})$
is isomorphic to one in
$\DAlg^+(\hat{\lngG})$.

To verify statement~(1) it suffices to check that
\begin{enumerate}
\item[(i)]
the object function of $\frak{P}$ is a bijection between
the class of objects of $\MSAlg^+(\lngG)$ and 
the class of objects of $\DAlg^+(\hat{\lngG})$,
\item[(ii)]
the morphism function of $\frak{P}$ is a bijection between
the set of homomorphisms $\al{M}\to\al{N}$ in $\MSAlg^+(\lngG)$
and the set of homomorphisms $\frak{P}(\al{M})\to\frak{P}(\al{N})$ in 
$\DAlg^+(\hat{\lngG})$ for all objects $\al{M},\al{N}$ in 
$\MSAlg^+(\lngG)$, and
\item[(iii)]
the inverses of these object and morphism functions yield a functor
$\DAlg^+(\hat{\lngG})\to\MSAlg^+(\lngG)$.
\end{enumerate}
For (i), the object function of $\frak{P}$ is surjective by the definition
of the category $\DAlg^+(\hat{\lngG})$, and it is injective, because
for any multisorted $\lngG$-algebra $\al{M}$, 
the $\hat{\lngG}$-algebra $\frak{P}(\al{M})$ uniquely determines
$\al{M}$. Indeed, since the underlying set $M\us{1}\times\dots\times M\us{m}$
of $\frak{P}(\al{M})$ is nonempty, this product set determines the sorts
$M\us{1},\dots,M\us{m}$ of $\al{M}$, and for every $g\in\lngG$, the operation
$\hat{g}$ of $\frak{P}(\al{M})$ determines the operation $g$ of $\al{M}$.

For (ii), let us fix two objects $\al{M},\al{N}$ in 
$\MSAlg^+(\lngG)$. It follows easily from the definition 
$\zeta=(\zeta\us{i})_{i\in[m]}\mapsto\frak{P}(\zeta)
=\zeta\us{1}\times\dots\times\zeta\us{m}$
of the morphism function of $\frak{P}$ that
homomorphisms $\al{M}\to\al{N}$ are mapped into homomorphisms
$\frak{P}(\al{M})\to\frak{P}(\al{N})$ injectively.
To see that the morphism function is also surjective, notice that the diagonal
operation $\diag$ of the algebras in $\DAlg^+(\hat{\lngG})$
forces every homomorphism $\frak{P}(\al{M})\to\frak{P}(\al{N})$ to be of 
the form $\sigma\us{1}\times\dots\times\sigma\us{m}$ for some
functions $\sigma\us{i}\colon M\us{i}\to N\us{i}$. 
So, since such a homomorphism
also preserves the operations $\hat{g}$ ($g\in\lngG$), it follows that
$(\sigma\us{i})_{i\in[m]}$ must be a homomorphism $\al{M}\to\al{N}$.

The last statement, (iii), follows immediately from the definitions
of the object and morphism functions of $\frak{P}$, completing the proof.
\end{proof}

Later on in the paper, when we use the definitions of the operations
of the algebra $\frak{P}(\al{M})$ (see Proposition~\ref{prp-functorP})
it will be convenient to think of the elements of $\prod_{i\in[m]} M\us{i}$
as column vectors of length $m$, and $k$-tuples of elements of
$\prod_{i\in[m]} M\us{i}$ as $m\times k$ matrices whose columns are in  
$\prod_{i\in[m]} M\us{i}$.
Accordingly, we will introduce the following convention:
for an $m\times k$ matrix
$\al{a}=(a\lsus{j}{i})_{i\in[m],j\in[k]}\in\bigl(\prod_{i\in[m]} M\us{i}\bigr)^k$,
we will denote
the $j$-th column $(a\lsus{j}{i})_{i\in[m]}$ of $\al{a}$ by $\al{a}\ls{j}$,
and the $i$-th row $(a\lsus{j}{i})_{j\in[k]}$ of $\al{a}$ by $\al{a}\us{i}$. 
Thus, the definitions of the operations of $\frak{P}(\al{M})$ in
\eqref{eq-constrd} and \eqref{eq-constrg} can be rewritten as follows:
\begin{itemize}
  \item
for every $m\times m$ matrix
$\al{a}=(a\lsus{j}{i})_{i\in[m],j\in[m]}\in\bigl(\prod_{i\in[m]} M\us{i}\bigr)^m$,
\begin{equation*}
  \diag(\al{a})=\text{diagonal of $\al{a}$},
  \tag*{(\ref{eq-constrd})$'$}
\end{equation*}
\item
for every symbol $g\in\mathcal{G}$
with input and output sorts $\al{i}=(i_1,\dots,i_k)\in[m]^k$ and $i'\in[m]$
  (and hence of arity $k$),
and
for every $m\times k$ matrix
$\al{a}=(a\lsus{j}{i})_{i\in[m],j\in[k]}\in\bigl(\prod_{i\in[m]} M\us{i}\bigr)^k$,
\begin{equation*}
  \hat{g}(\al{a})\ \text{is obtained from $\al{a}\ls{1}$ by replacing
    the $i'$-th entry with $g(a\lsus{1}{i_1},\dots,a\lsus{k}{i_k})$}.
  \tag*{(\ref{eq-constrg})$'$}
\end{equation*}  
\end{itemize}

\begin{remrk}
  \label{rm-terms-G}
This description of the operations of $\mathcal{D}^+(\hat{\lngG})$
generalizes to all terms, and yields an alternative view of (and proof for)
Theorem~\ref{thm-functorP}. The description is as follows:

\smallskip
\noindent
{\it
For any positive integer $k$,
there is a bijection
$\displaystyle  T(\al{x})\mapsto
  \begin{bmatrix}
    t\us{1}(\al{x})\\
    \vdots\\
    t\us{m}(\al{x})
  \end{bmatrix}
$  
between
\begin{itemize}
\item
  the $k$-ary terms $T$ of $\mathcal{D}^+(\hat{\lngG})$
(modulo the identities in $\mathcal{D}^+(\hat{\lngG})$)
  and
\item
  the $m$-tuples $t\us{1},\dots,t\us{m}$ 
of $mk$-ary $\lngG$-terms such that the input sort of each $t\us{i}$
$(i\in[m])$
is the $m\times k$ matrix with $j$-th row $[j\ \dots\ j]$ for each
$j\in[m]$, while the output sort is $i$
\end{itemize}
such that
the following equality holds in the algebra
$\al{P}:=\frak{P}(\al{M})$ for each object
$\al{M}=\bigl((M^{(i)})_{i\in[m]},\lngG\bigr)\in\MSAlg^+(\lngG)$:
  \begin{equation*}
    \qquad
  T(\al{a})=
  \begin{bmatrix}
    t\us{1}(\al{a})\\
    \vdots\\
    t\us{m}(\al{a})
  \end{bmatrix}
  \quad
  \text{for every $m\times k$ matrix
    $\al{a}\in\Bigl(\prod_{i\in[m]}M^{(i)}\Bigr)^k=P^k$,}
  \end{equation*}
  where $T$ is applied to the $k$ columns of $\al{a}$, 
  and $t\us{1},\dots,t\us{m}$ are applied to the $mk$ entries of $\al{a}$.}

\smallskip
\noindent
It is clear from the definitions of the operations in
$\mathcal{D}^+(\hat{\lngG})$
that terms of height $\le 1$ have the form
described in the statement, and
an easy induction on the complexity of terms yields the same result for
arbitrary terms $T$. For the converse, one can prove first,
again by induction on the complexity of terms, that 
every $\lngG$-term with output sort $i$ occurs as the $i$-th coordinate term
$t^{(i)}$ of some term $T$ of $\mathcal{D}^+(\hat{\lngG})$, and then use
the diagonal operation $\diag$ to construct a single term $T$ as claimed
for every sequence
$t\us{1},\dots,t\us{m}$  of $\lngG$-terms satisfying the requirements in the
statement above.
\end{remrk}

This paper will focus on the composition of the functors
$\frak{M}$ and $\frak{P}$. For further reference we now summarize the
consequences of Theorems~\ref{thm-functorM} and \ref{thm-functorP}
that we will need.

\begin{cor}\label{cor-functorC}
  For any algebraic language $\mathcal{F}$ with no nullary symbols
  and for any finite $\mathcal{F}$-algebra
  $\al{I}=([m];\mathcal{F})$, the category $\comma$
  of $\lngF$-algebras is equivalent to the category
  $\mathcal{D}^+(\hat{\lngF_{\al{I}}})$.

  In more detail, the functor $\frak{P}\circ\frak{M}$ yields an 
  isomorphism
  \[
  \frak{C}\colon\comma\to\DAlg^\boxplus(\hat{\lngF_{\al{I}}})
  \]
  between $\comma$ and the full subcategory
  $\DAlg^\boxplus(\hat{\lngF_{\al{I}}})$ of
  $\mathcal{D}^+(\hat{\lngF_{\al{I}}})$
  with object and morphism functions defined as follows:
\begin{itemize}
  \item
  $\frak{C}$ assigns to every object $(\al{A},\chi)$ in $\comma$
  the algebra 
  \[
  \frak{C}(\al{A},\chi):=
  \bigl(D\us{1}\times\dots\times D\us{m};\hat{\lngF}_{\al{I}}\bigr)
  \]
  in $\DAlg^\boxplus(\hat{\lngF}_{\al{I}})$ where 
  \begin{itemize}
  \item
    $\prod_{i\in[m]} D\us{i}$ is the product
    of the congruence classes $D\us{i}:=\chi^{-1}(i)$
    of the kernel of $\chi$,
  \item
    $\diag$ is the diagonal operation on $\prod_{i\in[m]} D\us{i}$, that is,
    for every $m\times m$ matrix
    $\al{a}=(a\lsus{j}{i})_{i\in[m],j\in[m]}\in\bigl(\prod_{i\in[m]} D\us{i}\bigr)^m$,
    \begin{equation*}
     \diag(\al{a})=\text{diagonal of $\al{a}$},
    \end{equation*}
    and
  \item
    for each $f\in\mathcal{F}$ (say $k$-ary), $\al{i}=(i_1,\dots,i_k)\in[m]^k$,
    and for every
    $m\times k$ matrix
    $\al{a}=(a\lsus{j}{i})_{i\in[m],j\in[k]}\in\bigl(\prod_{i\in[m]} D\us{i}\bigr)^k$,
    \begin{multline*}
      \qquad\qquad\qquad
    \hat{f}_{\wec{i}}(\al{a})\ \text{is obtained from $\al{a}\ls{1}$ by replacing
      the $f(\wec{i})$-th entry}\\
    \text{with $f(a\lsus{1}{i_1},\dots,a\lsus{k}{i_k})$, where
    $f(\wec{i})$ is evaluated in $\al{I}$}.
      \qquad
    \end{multline*}  
  \end{itemize}
  \item
  $\frak{C}$ assigns to every morphism
  $\psi=(\al{A},\chi)\to(\al{B},\xi)$ in $\comma$ the homomorphism
  \[
  \frak{C}(\psi):=\psi\us{1}\times\dots\times\psi\us{m}\colon
  \frak{C}(\al{A},\chi)\to\frak{C}(\al{B},\xi)
  \]
  of $\hat{\lngF_{\al{I}}}$-algebras,
  where for each $i\in [m]$, $\psi\us{i}\colon \chi^{-1}(i)\to\xi^{-1}(i)$
  is the restriction of $\psi$ to $\chi^{-1}(i)$, and
  $\psi\us{1}\times\dots\times\psi\us{m}$ is the induced function
  \[
  \chi^{-1}(1)\times\dots\times\chi^{-1}(m)\to
  \xi^{-1}(1)\times\dots\times\xi^{-1}(m).\]
  \end{itemize} 
\end{cor}

Note that it is necessary to assume that 
the algebra $\al{I}$ is finite
if we want to obtain a finitary diagonal operation $\diag$. The assumption that
the language $\lngF$ has no nullary symbols is also unavoidable, 
unless we are willing to lose nullary symbols during the construction
$(\al{A},\chi)\mapsto \frak{C}(\al{A},\chi)$.
Indeed, if an algebra $\frak{C}(\al{A},\chi)$ had a 
nullary operation symbol, and hence a corresponding unary constant 
term operation, then by the description of the term operations of 
$\frak{C}(\al{A},\chi)$ in
Corollary~\ref{cor-termops}(1),
the algebra $\al{A}$ should have constant term operations 
$t\us{i}$ (and hence their nullary counterparts) 
with values 
in every kernel class $\chi^{-1}(i)$ ($i\in[m]$)
of the homomorphism $\chi\colon\al{A}\to\al{I}$.

This restriction on the language $\lngF$ does not affect the applicability
of the functor $\frak{C}$. Indeed, if $\lngF$ is any language with at least
one nullary symbol, then letting $\lngF'$ be the algebraic language
obtained from $\lngF$ by replacing each nullary symbol $c\in F$ by a unary
symbol $c'$ we get a functor
\[
'\colon\Alg(\lngF)\to\Alg(\lngF')
\]
\begin{itemize}
\item
by assigning to every $\lngF$-algebra $\al{A}=(A;\lngF)$
(which is necessarily nonempty) the algebra $\al{A}'=(A,\lngF')$
obtained from $\al{A}$ by defining $c'$ to be the unary constant operation
on $A$ with value $c$ for every nullary symbol $c\in\lngF$, and by keeping all
other operations unchanged; and
\item
  by assigning to every homomorphism $\phi\colon\al{A}\to\al{B}$ of
  $\lngF$-algebras the same function $\phi':=\phi$, which is a homomorphism
  $\al{A}'\to\al{B}'$ of $\lngF'$-algebras.
\end{itemize}
Clearly, this is an isomorphism between $\Alg(\lngF)$ and a full subcategory
of $\Alg(\lngF')$.
Given a finite $\lngF$-algebra $\al{I}=([m];\lngF)$, this isomorphism
induces a functor
\[
\frak{I}\colon\comma\to\bigl(\Alg(\lngF')\,{\downdownarrows}\,\al{I}'\bigr)
\]
which is an isomorphism between $\comma$ and a full subcategory of
$\bigl(\Alg(\lngF')\,{\downdownarrows}\,\al{I}'\bigr)$.
This can now be composed with the isomorphism from Corollary~\ref{cor-functorC}
with domain category $\bigl(\Alg(\lngF')\,{\downdownarrows}\,\al{I}'\bigr)$.

  \begin{remrk}
    \label{rm-terms-FI}
For every algebraic language $\lngF$ with no nullary symbols
  and for any finite $\mathcal{F}$-algebra
  $\al{I}=([m];\mathcal{F})$,
  the description (stated without proof in Remark~\ref{rm-terms-G})
  for the terms of the variety
  $\mathcal{D}^+(\hat{\lngF_{\al{I}}})$ via $\lngF_{\al{I}}$-terms easily yields
  a description of the terms of
  $\mathcal{D}^+(\hat{\lngF_{\al{I}}})$ via $\lngF$-terms and $\al{I}$.
  In Section~\ref{sec-functorC} we give a proof for this result, but instead of
  using computations with terms, we derive the result from
  Corollary~\ref{cor-functorC} via examining the
  relationship between free algebras in $\comma$ and in $\Alg(\lngF)$
  (see Corollary~\ref{cor-terms}). 
\end{remrk}

\section{Algebraic properties of the functor $\frak{C}$}
\label{sec-functorC}

Throughout this section, $\lngF$ will be an algebraic language with no nullary
symbols, and $\al{I}=([m];\lngF)$ will be a finite $\lngF$-algebra ($m>0$). 
The functor $\frak{C}$ (see Corollary~\ref{cor-functorC})
describes a construction which produces from every
$\lngF$-algebra $\al{A}$ with a fixed homomorphism $\chi$
onto $\al{I}$ a new
algebra $\frak{C}(\al{A},\chi)$
in the language $\hat{\lngF}_{\al{I}}$.
Our goal in this section is to demonstrate
that this construction preserves
a lot of important algebraic properties.
As discussed in the Introduction, variants and special cases of many of these
preservation results have appeared in the literature, but often not in the
generality that we need in Sections~\ref{sec-appl1}--\ref{sec-appl2}.
Therefore we decided to include this section to survey a wide range
of algebraic properties preserved by the functor $\frak{C}$, most of which will
be used in the proofs of the main results of the paper in
Sections~\ref{sec-appl1}--\ref{sec-appl2}.

Our survey will start with
preservation theorems for subalgebras of products
(see Theorem~\ref{thm-subprod}) and
compatible relations, including endomorphisms and automorphisms
  (see   Corollaries~\ref{cor-A-n-alpha}--\ref{cor-endo}).
Of particular interest are the preservation results on
congruences (see Corollary~\ref{cor-congr}),
commutator properties (see Theorem~\ref{thm-comm}),
and tame congruence theoretic types (see Theorem~\ref{thm-tct}).
We will also describe how the clone of term (resp., polynomial) operations
changes when we pass from $\al{A}$ (paired with $\chi$) to $\frak{C}(\al{A},\chi)$
(see Corollaries~\ref{cor-termops}(1) and \ref{cor-poly}).
The description of term operations will imply,
for example, that under this  construction, 
the satisfaction of idempotent Maltsev conditions by the generated varieties
is also preserved (see Corollary~\ref{cor-clone-idem}).
  Finiteness properties of algebras and clones preserved by
  $\frakC$ will also be considered
  (see, e.g., Theorem~\ref{thm-subprod}(4) and
  Corollaries~\ref{cor-quotient}(3), \ref{cor-clone-fingen},
  \ref{cor-presentations}).

\subsection{Subalgebras of products and compatible relations}
\label{subsec-subalg}
In the category\break $\comma$, the subobjects of an object $(\al{A},\chi)$
are (up to isomorphism) the objects $(\al{B},\xi)$
such that $B\subseteq A$ and the inclusion map $B\to A$ is a morphism
$(\al{B},\xi)\to(\al{A,\chi})$; equivalently,
the subobjects of $(\al{A},\chi)$
are (up to isomorphism) the pairs $(\al{B},\xi)$ where
$\al{B}$ is a subalgebra of $\al{A}$ and
$\xi=\chi|_{\al{I}}\colon \al{B}\to\al{I}$ is onto
(i.e., $B$ has a nonempty intersection with $\chi^{-1}(i)$
for all $i\in[m]$).

The category $\comma$ has products for all nonempty families of objects.
Namely, if 
$(\al{A}\ls{j},\chi\ls{j})$ ($j\in J$, $J\not=\emptyset$)
is an indexed family of objects in $\comma$,
  then their product in $\comma$ is their pullback in $\Alg(\lngF)$, one
  representative of which
is the fibered product $\prodI_{j\in J}\al{A}\ls{j}$
together with the induced homomorphism $\chi$ defined as follows:
\[
\ProdI_{j\in J}A\ls{j}
=\Bigl\{(a\ls{j})_{j\in J}\in\prod_{j\in J}A\ls{j}:
\chi\ls{j}(a\ls{j})=\chi\ls{j'}(a\ls{j'}) \text{ for all $j,j'\in J$}
\Bigr\}
=\bigcup_{i\in[m]}\Bigl(\prod_{j\in J}D\lsus{j}{i}\Bigr)
\]
where $D\lsus{j}{i}=\chi^{-1}\ls{j}(i)$ for all $i\in[m]$, $j\in J$, and
\[
\chi\colon \ProdI_{j\in J}\al{A}\ls{j}\to\al{I},\quad
(a\ls{j})_{j\in J}\mapsto \chi\ls{j}(a\ls{j}),
\]
where the image $\chi\ls{j}(a\ls{j})$ is independent of the choice of $j\in J$.

\begin{thm}
\label{thm-subprod}
Let $(\al{A}\ls{j},\chi\ls{j})$ $(j\in J)$ be arbitrary
  objects of $\comma$,
  let $\al{C}_j:=\frak{C}(\al{A}\ls{j},\chi\ls{j})$ $(j\in J)$,
  and for any $i\in[m]$ let
  $D\us{i}:=\prod_{l\in J}\chi\ls{l}^{-1}(i)$.
  \begin{enumerate}
    \item[{\rm(1)}]
   $\frak{C}\bigl(\prodI_{j\in J}\al{A}\ls{j},\chi\bigr)=\prod_{j\in J}\al{C}_j$,
      and the mapping
      \begin{align*}
      B\mapsto  B^\frakC & {}:=
        \left\{
    \begin{pmatrix}
      \begin{bmatrix}
        b\lsus{j}{1}\\
        \vdots\\
        b\lsus{j}{m}
      \end{bmatrix}
    \end{pmatrix}_{\kern-5pt j\in J}\in\prod_{j\in J}C_j:
    (b\lsus{j}{i})_{j\in J}\in B\cap D\us{i}\ 
    \text{for all $i\in[m]$}
    \right\}\\
    & {} = \prod_{i\in[m]}(B\cap D\us{i})
    \quad \text{(up to a rearrangement of coordinates)}
    \end{align*}
    is an isomorphism between
    \begin{enumerate}
    \item[$\diamond$]
      the ordered set of all subuniverses $B$ of
      the algebra $\prodI_{j\in J}\al{A}\ls{j}$ with
      the property that $B\cap D\us{i}\not=\emptyset$ for every $i\in[m]$,
      and
    \item[$\diamond$]
      the ordered set of all nonempty subuniverses of
      the algebra $\prod_{j\in J}\al{C}_j$.
    \end{enumerate}
    
  \item[{\rm(2)}]
  This mapping\,\ $^\frakC$ 
  preserves all intersections of subuniverses 
  and all coordinate manipulations on subuniverses
  (namely: permutation and duplication of coordinates, and
  projection of subuniverses onto subsets of coordinates).
  In more detail, for intersection and for projection this means that
  if $R\ls{\ell}$ $(\ell\in L)$ and $R$ are subuniverses of
  $\prodI_{j\in J}\al{A}_j$ in the domain
  of\,\ ${}^\frakC$ and $K\subseteq J$, then
  \[
  \qquad
  \Bigl(\bigcap_{\ell\in L} R\ls{\ell}\Bigr)^\frakC=\bigcap_{\ell\in L}R\ls{\ell}^\frakC
  \quad \text{and}\quad
  \bigl(R|_K\bigr)^\frakC=R^\frakC|_K,
  \]
  respectively, where $=$ means that if the left hand side is defined,
  then so is the right hand side and the equality holds.
    
  \item[{\rm(3)}]
    For arbitrary choice of elements
    $d\us{i}=(d\lsus{j}{i})_{j\in J}\in D\us{i}$ $(i\in[m])$, the
    function
    \[
    \tilde{\phantom{G}}\colon \ProdI_{j\in J} A \ls{j}\to\prod_{j\in J}C_j,\quad
    x=(x_j)_{j\in J}\ \ \mapsto\ \ 
    \tilde{x}:=
    \begin{pmatrix}
      \begin{bmatrix}
        d\lsus{j}{1}\\
        \vdots\\
        d\lsus{j}{i-1}\\
        x\ls{j}\\
        d\lsus{j}{i+1}\\
        \vdots\\
        d\lsus{j}{m}
      \end{bmatrix}
    \end{pmatrix}_{\kern-5pt j\in J}
    \ \ \text{if $x\in D\us{i}$}
    \]
    has the following property:
    If $\al{B}$ is a subalgebra of $\prodI_{j\in J}\al{A}\ls{j}$ such
    that $d\us{i}\in B\cap D\us{i}$ for every $i\in[m]$,
    then for every set $G\subseteq B$
    that generates $\al{B}$,
    the set $\tilde{G}:=\{\tilde{b}:b\in G\}$
    generates $\al{B}^\frakC$.

  \item[{\rm(4)}]
    Consequently, for every subalgebra $\al{B}$ of
    $\prodI_{j\in J}\al{A}\ls{j}$ such that
    $B\cap D\us{i}\not=\emptyset$ for every $i\in[m]$,
    $\al{B}$ is finitely generated if and only if
    $\al{B}^\frakC$ is finitely generated.
  \end{enumerate}  
\end{thm}  

The elements $d\us{i}$ ($i\in[m]$) used in the definition of the function
$\tilde{\phantom{G}}$ in part (3) of the theorem above will be referred to as \emph{padding elements}.

\begin{proof}[Proof of Theorem~\ref{thm-subprod}]
  Statements (1)--(2)
  follow by combining our description of subobjects
    and products in the category $\comma$ (see the first two paragraphs
    of this subsection) with the following properties of the functor 
    $\frak{C}$ (see Corollary~\ref{cor-functorC}):
    (i)~$\frak{C}$ is an isomorphism between the category $\comma$ and its
    image category $\frak{C}\comma$, (ii)~$\frak{C}$ sends the product object
  $\Bigl(\prodI_{j\in J}\al{A}\ls{j},\chi\bigr)$ to the product algebra
  $\prod_{j\in J}\frak{C}(\al{A}_j,\chi_j)=\prod_{j\in J}\al{C}_j$,
  and (iii)~$\frak{C}$ sends
  the subobjects $(\al{B},\chi|_\al{B})$ of an object
  $(\al{A},\chi)$ to the subalgebras
  $\al{B}^\frakC:=\frak{C}(\al{B},\chi|_\al{B})$ of
  $\frak{C}(\al{A},\chi)$.

  For the proof of statement (3)
  assume $G$ generates $\al{B}$, and
  let $\al{T}$ denote the subalgebra of
  $\prod_{j\in J}\al{C}_j$ generated by the set
  $\tilde{G}$. 
  Since $\al{T}$ is nonempty (as $\al{B}$,
  and hence $G$ are nonempty), statement~(1) implies that 
  $\al{T}=\al{S}^\frakC$ for some subalgebra $\al{S}$ of
  $\prodI_{j\in J}\al{A}\ls{j}$ 
  such
  that $S\cap D\us{i}\not=\emptyset$ for every $i\in[m]$.
  Our goal is to show that $\al{T}=\al{B}^\frakC$, or equivalently,
  $\al{S}^\frakC=\al{B}^\frakC$. 
  
  Since $G\subseteq B$, it follows from the description of the elements of
  $\al{B}^\frakC$ that $\tilde{G}\subseteq B^\frakC$. Therefore,
  $\al{S}^\frakC=\al{T}$ is a subalgebra of $\al{B}^\frakC$.
  To prove that $\al{B}^\frakC$ is a subalgebra of $\al{S}^\frakC$, it suffices to
  establish that $\al{B}$ is a subalgebra of $\al{S}$. Since
  $\tilde{G}\subseteq\al{S}^\frakC$, the description of the elements of $\al{S}^\frakC$
  implies that $G\cup\{d\us{i}:i\in[m]\}\subseteq S$. Hence
  the algebra $\al{B}$ generated by $G$ is a subalgebra of $\al{S}$.

    In statement~(4), the implication that if $\al{B}$ is finitely generated,
    then so is $\al{B}^\frakC$ follows from the statement in part (3). Since
    in that statement $G$ is finite if and only if $\tilde{G}$ is finite
    if and only if $G\cup\{d^{(i)}:i\in[m]\}$ is finite, the converse follows
    by using that $\frak{C}$ is an isomorphism $\comma\to\frak{C}\comma$.
\end{proof}

An important special case of Theorem~\ref{thm-subprod} is when
all objects $(\al{A}_j,\chi_j)$ ($j\in J$) are the same, say $(\al{A},\chi)$.
Then the algebra $\prodI_{j\in J}\al{A}$ depends only on $\al{A}$ and
$\alpha=\ker(\chi)$, and will be denoted by $\al{A}^J[\alpha]$.
The underlying set is
\[
A^J[\alpha]:=\{(a\ls{j})_{j\in J}:a\ls{j}\,\alpha\,a\ls{\ell}
\text{ for all $j,\ell\in J$}\}.
\]
In particular, $A^2[\alpha]$ is nothing else than $\alpha$.

Recall that for an algebra $\al{A}$ the subuniverses of $\al{A}^n$
are often called $n$-ary \emph{compatible relations} of $\al{A}$.
An $n$-ary compatible relation of $\al{A}$, and the corresponding algebra,
is said to be \emph{reflexive} if it contains
all constant $n$-tuples $(a,\dots,a)$ ($a\in A$).
The map $^\frakC$ described in Theorem~\ref{thm-subprod}(1) is particularly
well-behaved for reflexive compatible relations of $\al{A}$
that are contained in $\al{A}^n[\alpha]$, as the following corollary shows.

\begin{cor}
\label{cor-A-n-alpha}
Let $\al{A}$ be an $\mathcal{F}$-algebra
with an onto homomorphism $\chi\colon\al{A}\to\al{I}$.
Let $\al{C}:=\frak{C}(\al{A},\chi)$ and
$\alpha:=\ker(\chi)$.
\begin{enumerate}
\item[{\rm(1)}]
  For every positive integer $n$, the mapping 
  \begin{align*}
    \qquad
    B &{}\mapsto  B^\frakC\\
      &{}=
        \left\{
    \begin{pmatrix}
      \begin{bmatrix}
        a\lsus{1}{1}\\
        \vdots\\
        a\lsus{1}{m}
      \end{bmatrix},\dots,
      \begin{bmatrix}
        a\lsus{n}{1}\\
        \vdots\\
        a\lsus{n}{m}
      \end{bmatrix}
    \end{pmatrix}\in\al{C}^n:
    (a\lsus{1}{i},\dots,a\lsus{n}{i})\in B\ 
    \text{for all $i\in[m]$}
    \right\}
  \end{align*}
  is an isomorphism between
    \begin{enumerate}
    \item[$\diamond$]
      the lattice of all $n$-ary reflexive compatible relation of $\al{A}$
      that are contained in $A^n[\alpha]$,
      and
    \item[$\diamond$]
      the lattice of all $n$-ary reflexive compatible relations of
      $\al{C}$.
    \end{enumerate}  
\item[{\rm(2)}]
  This mapping $^\frakC$ between reflexive relations
  (applied over all arities $n\ge1$) also preserves
  the following relational clone operations: intersection of relations
  of the same arity, composition, and coordinate manipulations
  (permutation, identification, and duplication of coordinates, and
  projection of relations to a subset of coordinates).
  Moreover, $^\frakC$ preserves product in the respective categories; that is,
  if for any reflexive
  compatible relations $R\subseteq A^k[\alpha]$ and
  $S\subseteq A^n[\alpha]$ we define
  $R\times_\alpha S:=(R\times S)\cap A^{k+n}[\alpha]$, then
  \[
(R\times_\alpha S)^\frakC=R^\frakC\times S^\frakC.
  \]
  \end{enumerate}  
\end{cor}

\begin{proof}
  Statement~(1) follows immediately from Theorem~\ref{thm-subprod}(1)
  and the fact that for any $n$-ary reflexive relation $B$ of $\al{A}$
  with $B\subseteq A^n[\alpha]$ the property that
  ``$B\cap \bigl(\chi^{-1}(i)\bigr)^n\not=\emptyset$ for every $i\in[m]$''
  holds automatically, since $B\cap \bigl(\chi^{-1}(i)\bigr)^n$ contains all
  constant tuples
  $(a,\dots,a)$ with $a\in\chi^{-1}(i)$.
  Statement~(2) is a straightforward consequence of statement (1). 
\end{proof}

Recall that by Corollary~\ref{cor-functorC},
every homomorphism from $\frak{C}(\al{A},\chi)$ to $\frak{C}(\al{B},\xi)$
is of the form $\frak{C}(\psi)$ for
some homomorphism $\psi\colon \al{A}\to \al{B}$ with $\chi = \xi\circ\psi$.
By identifying endomorphisms $\psi$ of $\al{A}$ with their graphs
$\{ (a,\psi(a)) : a\in A\}$, which form subuniverses
of $\al{A}^2$, the map $^\frakC$ described in Theorem~\ref{thm-subprod}(1)
also yields the following correspondence.

\begin{cor}
\label{cor-endo}
Let $\al{A}$ be an $\mathcal{F}$-algebra
with an onto homomorphism $\chi\colon\al{A}\to\al{I}$.
Let $\al{C}:=\frak{C}(\al{A},\chi)$ and
$\alpha:=\ker(\chi)$.
  The mapping 
\begin{equation*}
  \qquad
      \psi\mapsto  \psi^\frakC =
        \left\{
    \begin{pmatrix}
      \begin{bmatrix}
        a\us{1}\\
        \vdots\\
        a\us{m}
      \end{bmatrix},
      \begin{bmatrix}
        \psi(a\us{1})\\
        \vdots\\
        \psi(a\us{m})
      \end{bmatrix}
    \end{pmatrix}\in\al{C}^2:
    a\us{i}\in\chi^{-1}(i)\ 
    \text{for all $i\in[m]$}
    \right\}
\end{equation*}
  is an isomorphism between
    \begin{enumerate}
    \item[$\diamond$]
      the monoid of all endomorphisms (the group of all automorphisms) 
      $\psi$ of $\al{A}$ such that 
      the graph of $\psi$ is contained in $\alpha$,
      and
    \item[$\diamond$]
      the endomorphism monoid (the automorphism group) of
      $\al{C}$.
    \end{enumerate}  
\end{cor}

\subsection{Congruences and commutator properties}

  Results of the previous subsection specialize
  to congruences as follows.

\begin{cor}
\label{cor-congr}
Let $\al{A}$ be an $\mathcal{F}$-algebra
with an onto homomorphism $\chi\colon\al{A}\to\al{I}$, and let
$\al{C}:=\frak{C}(\al{A},\chi)$, $\alpha:=\ker(\chi)$.
\begin{enumerate}
\item[{\rm(1)}]
The mapping
\begin{equation*}
  \qquad
      \beta\mapsto  \beta^\frakC =
        \left\{
    \begin{pmatrix}
      \begin{bmatrix}
        a\lsus{1}{1}\\
        \vdots\\
        a\lsus{1}{m}
      \end{bmatrix},
      \begin{bmatrix}
        a\lsus{2}{1}\\
        \vdots\\
        a\lsus{2}{m}
      \end{bmatrix}
    \end{pmatrix}\in\al{C}^2:
    (a\lsus{1}{i},a\lsus{2}{i})\in\beta\ 
    \text{for all $i\in[m]$}
    \right\}
\end{equation*}
is an isomorphism between the interval $I(0,\alpha)$ of
the congruence lattice of $\al{A}$ and the congruence lattice of
$\al{C}$.
\item[{\rm(2)}]
The mapping ${}^\frakC$ preserves meet (= intersection) and join of congruences,
and $k$-permutability of congruences for every $k\ge2$; that is,
for arbitrary congruences $\beta_j\in I(0,\alpha)$ $(j\in J)$
and $\beta,\gamma\in I(0,\alpha)$ of $\al{A}$,
\[
\Bigl(\bigcap_{j\in J}\beta_j\Bigr)^\frakC=\bigcap_{j\in J}\beta_j^\frakC,
\qquad
\Bigl(\bigvee_{j\in J}\beta_j\Bigr)^\frakC=\bigvee_{j\in J}\beta_j^\frakC,
\]
and
\[
\beta\circ_k\gamma=\gamma\circ_k\beta
\ \ \Longleftrightarrow\ \ 
\beta^\frakC\circ_k\gamma^\frakC=\gamma^\frakC\circ_k\beta^\frakC.
\]
\item[{\rm(3)}]
  For every $\beta\in I(0,\alpha)$, $\beta$ is a finitely generated
  congruence of $\al{A}$ if and only if $\beta^\frakC$ is a finitely generated
  congruence of $\al{C}$.
\end{enumerate}
\end{cor}  

\begin{proof}
  Statements~(1)--(2) are immediate consequences of
  Corollary~\ref{cor-A-n-alpha}.
  Statement (3) follows from (2) and the fact that finitely generated
  congruences are exactly the compact elements in the congruence lattice.
\end{proof}

The description of the functor $\frak{C}$ in Corollary~\ref{cor-functorC},
together with the description of the mapping ${}^\frakC$ it induces
on congruences, implies that
quotient algebras are also preserved by $\frak{C}$ 
in the following sense.

\begin{cor}
  \label{cor-quotient}
Let $\al{A}$ be an $\mathcal{F}$-algebra
with an onto homomorphism $\chi\colon\al{A}\to\al{I}$, and let
$\alpha:=\ker(\chi)$.
\begin{enumerate}
\item[{\rm(1)}]
For every congruence $\beta\in I(0,\alpha)$ of $\al{A}$,
\[
\frak{C}(\al{A}/\beta,\chi/\beta)\cong \frak{C}(\al{A},\chi)/\beta^\frakC,
\]
where $\chi/\beta\colon\al{A}/\beta\to\al{I}$ is the unique homomorphism
that is obtained by factoring
$\chi\colon\al{A}\to\al{I}$ through the natural homomorphism
$\al{A}\to\al{A}/\beta$.
\item[{\rm(2)}]
  In particular, for every congruence $\beta\in I(0,\alpha)$ of $\al{A}$,
  $\al{A}/\beta$ is finite if and only if
  $\frak{C}(\al{A},\chi)/\beta^\frakC$ is finite.
  Consequently, $\al{A}$ is residually finite if and only if
  $\frak{C}(\al{A},\chi)$ is residually finite.
\end{enumerate}
\end{cor}

  Recall (see, e.g., \cite{bulHC, AM:SAHC, Mo:HCT})
  that for any algebra $\al{A}$, the $k$-ary commutator
  is a $k$-ary operation $[{-},\dots,{-}]$
  on the congruence lattice of $\al{A}$,
  and for any choice of $\beta_1,\dots,\beta_k\in\Con(\al{A})$,
  the definition of the commutator $[\beta_1,\dots,\beta_k]$
  uses
  a specific subalgebra $\al{M}_{\al{A}}(\beta_1,\dots,\beta_k)$
  of $\al{A}^{{\mathbf2}^k}$, called the algebra of
  \emph{$(\beta_1,\dots,\beta_k)$-matrices}.
  It is useful to think of the elements of $A^{\mathbf{2}^k}$ as functions
  $f\colon {\mathbf{2}}^k\to A$ labeling the vertices of the
  $k$-dimensional cube $\mathbf{2}^k=\{0,1\}^k$.
  For each $j\in[k]$, $\mathbf{2}^k$ has two `faces' of codimension $1$
  perpendicular to the $j$-th direction: $F_j(0)$ and $F_j(1)$, where
  $F_j(u)$ is the set of all elements of $\mathbf{2}^k$ with $j$-th
  coordinate $u$ ($u\in\mathbf{2}$).
  The algebra $\al{M}_{\al{A}}(\beta_1,\dots,\beta_k)$ of
  $(\beta_1,\dots,\beta_k)$-matrices is generated by all labelings
  of the vertices of
  the $k$-dimensional cube that are constant on the faces
  $F_j(0)$ and $F_j(1)$ for some $j\in[k]$, and have the property that
  the two values they assume are $\beta_j$-related. 
  We introduce some notation to describe this
  generating set.
  For any $j\in[k]$ and $a_0\,\beta_j\,a_1$,
  let $g_j[a_0,a_1]$ denote the labeling $g\colon\mathbf{2}^k\to A$ such that
  for each $u\in\mathbf{2}$ we have that $g|_{F_j(u)}$ is constant with value
  $a_u$. So, the standard generating set for
  $\al{M}_{\al{A}}(\beta_1,\dots,\beta_k)$ is
  \begin{equation}
    \label{eq-HCgenset-A}
  G=\bigcup_{j\in[k]} G_j
  \quad\text{where}\quad
  G_j:=\{g_j[a_0,a_1]:a_0\,\beta_j\,a_1\}
  \ \ \text{for all $j\in[k]$}.
  \end{equation}
  The commutator $[\beta_1,\dots,\beta_k]$ is defined to be the least
  congruence $\gamma$ of $\al{A}$ with the following property:
  \begin{enumerate}
  \item[$(*)$]
    whenever $f$ is an element of $\al{M}_{\al{A}}(\beta_1,\dots,\beta_k)$
    such that $f(\epsilon0)\,\gamma\,f(\epsilon1)$ for all
    $\epsilon\in\mathbf{2}^{k-1}\setminus\{1\dots1\}$, then
    we also have that  
    $f(1\dots10)\,\gamma\,f(1\dots11)$.
  \end{enumerate}  
  Note that condition $(*)$ holds for $\gamma=\beta_k$, because
  every labeling $f$ in $G$ --- and
  hence also every labeling $f$ in $\al{M}_{\al{A}}(\beta_1,\dots,\beta_k)$ ---
  satisfies $f(\epsilon0)\,\beta_k\,f(\epsilon1)$ for all
  $\epsilon\in\mathbf{2}^{k-1}$. Therefore,
  $[\beta_1,\dots,\beta_k]\le\beta_k$.

\begin{thm}
  \label{thm-comm}
Let $\al{A}$ be an $\mathcal{F}$-algebra
with an onto homomorphism $\chi\colon\al{A}\to\al{I}$, and let
$\al{C}:=\frak{C}(\al{A},\chi)$,
$\alpha:=\ker(\chi)$.
The isomorphism ${}^\frakC$ between the interval $I(0,\alpha)$ of
the congruence lattice of $\al{A}$ and the congruence lattice of
$\al{C}$ (see Corollary~\ref{cor-congr}(1))
preserves higher commutators of congruences;
that is, for any integer $k\ge2$
and for arbitrary congruences $\beta_1,\dots,\beta_k\in I(0,\alpha)$
of $\al{A}$,
\begin{equation}
  \label{eq-HC}
[\beta_1,\dots,\beta_k]^\frakC=[\beta_1^\frakC,\dots,\beta_k^\frakC].
\end{equation}
\end{thm}  

\begin{proof}
  We start by examining the effect of the mapping ${}^\frakC$ from
  Corollary~\ref{cor-A-n-alpha} to the algebra of 
  $(\beta_1,\dots,\beta_k)$-matrices for congruences
  $\beta_1,\dots,\beta_k$ below $\alpha$.

  \begin{clm}
    \label{clm-matrix-algs}
    Under the same assumptions on $\al{A}$, $\chi$, and $\beta_1,\dots,\beta_k$
    as in Theorem~\ref{thm-comm},
  \begin{equation}
    \label{eq-matrix-algs}
  \bigl(\al{M}_{\al{A}}(\beta_1,\dots,\beta_k)\bigr)^\frakC
  = \al{M}_{\al{C}}(\beta_1^\frakC,\dots,\beta_k^\frakC).
  \end{equation}
  \end{clm}

  \begin{proof}[Proof of Claim~\ref{clm-matrix-algs}]
  The standard generating set for
  the algebra $\al{M}_{\al{A}}(\beta_1,\dots,\beta_k)$ of
  $(\beta_1,\dots,\beta_k)$-matrices in $\al{A}$ is the set $G$
  indicated in \eqref{eq-HCgenset-A}. Analogously, the standard generating set
  for the algebra $\al{M}_{\al{C}}(\beta_1^\frakC,\dots,\beta_k^\frakC)$ of
  $(\beta_1^\frakC,\dots,\beta_k^\frakC)$-matrices in $\al{C}$ is the set
  \begin{equation}
    \label{eq-HCgenset-C}
  H=\bigcup_{j\in[k]} H_j
  \quad\text{where}\quad
  H_j:=\{h_j[c_0,c_1]:c_0\,\beta_j^\frakC\,c_1\}
  \ \ \text{for all $j\in[k]$},
  \end{equation}
  where for any elements $c_0,c_1\in C$ with $c_0\,\beta_j^\frakC\,c_1$,
  $h_j[c_0,c_1]$ denotes the labeling $h\colon\mathbf{2}^k\to C$
  of the vertices of the $k$-dimensional cube with elements of $\al{C}$
  in such a way that
  for each $u\in\mathbf{2}$, $h|_{F_j(u)}$ is constant with value
  $c_u$. 

  To see that the left hand side of \eqref{eq-matrix-algs}
  is defined, observe that
  $\al{M}_{\al{A}}(\beta_1,\dots,\beta_k)$ is a subalgebra of
  $\al{A}^{\mathbf{2}^k}[\alpha]=\prodI_{\epsilon\in\mathbf{2}^k}\al{A}$, 
  because the assumption
  $\beta_1,\dots,\beta_k\le\alpha$ ensures that
  $G\subseteq A^{\mathbf{2}^k}[\alpha]$.
  Moreover, $\al{M}_{\al{A}}(\beta_1,\dots,\beta_k)$ is reflexive,
  since $G$ contains all constant labelings $\mathbf{2}^k\to A$.
  Choose and fix one constant labeling
  $g_1[d\us{i},d\us{i}]\,(=\dots=g_k[d\us{i},d\us{i}])$
  with value
  $d\us{i}\,(\in\chi^{-1}(i))$ from each set $D\us{i}$ ($i\in[m]$).
  By Theorem~\ref{thm-subprod}(3), the algebra
  $\bigl(\al{M}_{\al{A}}(\beta_1,\dots,\beta_k)\bigr)^\frakC$ is generated by
  the set
  \begin{equation}
    \label{eq-HCgenset-A*}
  \tilde{G}=\bigcup_{j\in[k]} \tilde{G}_j
  \quad\text{where}\quad
  \tilde{G}_j:=\{h_j[\tilde{a}_0,\tilde{a}_1]:a_0\,\beta_j\,a_1\}
  \ \ \text{for all $j\in[k]$},
  \end{equation}
  because the image of each labeling $g_j[a_0,a_1]\in G$ under the function
  $\tilde{\phantom{G}}\colon \al{A}^{\bf{2}^k}\to \al{C}^{\bf{2}^k}$
  with padding elements $g_1[d\us{i},d\us{i}]$ ($i\in[m]$)
  is the labeling $h_j[\tilde{a}_0,\tilde{a}_1]$, where
  $\tilde{a}_0,\tilde{a}_1$ are obtained from $a_0,a_1$ by applying the
  function $\tilde{\phantom{G}}\colon \al{A}\to \al{C}$
  with padding elements $d\us{i}$ ($i\in[m]$).
  Since $\tilde{G}\subseteq H$, the inclusion $\subseteq$ in
  \eqref{eq-matrix-algs} follows.
  For the reverse inclusion notice that
  $\al{M}_{\al{C}}(\beta_1^\frakC,\dots,\beta_k^\frakC)$ is a reflexive subalgebra of
  $\al{C}^{\mathbf{2}^k}$, hence by Corollary~\ref{cor-A-n-alpha}(1),
  it is equal to $\al{B}^\frakC$ for some reflexive subalgebra $\al{B}$ of
  $\al{A}^{\mathbf{2}^k}[\alpha]$.
  Thus, $\al{B}^\frakC$ is the least reflexive subalgebra of $\al{C}^{\mathbf{2}^k}$
  which contains all generators
  $h_j[c_0,c_1]\in H$ of $\al{M}_{\al{C}}(\beta_1^\frakC,\dots,\beta_k^\frakC)$ from
  \eqref{eq-HCgenset-C}; so $j\in[k]$ and
  \begin{multline*}
    \qquad
  c_0=\begin{bmatrix}c_0\us{1}\\ \vdots\\ c_0\us{m}\end{bmatrix},\ \ 
  c_1=\begin{bmatrix}c_1\us{1}\\ \vdots\\ c_1\us{m}\end{bmatrix}\\
  \text{where}\ \
  c_0\us{i}\,\beta_j\,c_1\us{i} \ \ \text{and}\ \
  c_0\us{i},c_1\us{i}\in\chi^{-1}(i)
  \ \ \text{for each $i\in[m]$}.
  \qquad
  \end{multline*}
  It follows from Corollary~\ref{cor-A-n-alpha}(1)
  that $\al{B}$ is the least reflexive subalgebra of
  $\al{A}^{\mathbf{2}^k}[\alpha]$ which contains the projections of these
  generators to all coordinates $i\in[m]$, which are the functions
    \[
  g_j[c_0\us{i},c_1\us{i}]
    \ \ \text{with}\ \
    c_0\us{i}\,\beta_j\,c_1\us{i}\ \ \text{and}\ \
    c_0\us{i},c_1\us{i}\in\chi^{-1}(i).
  \]
  Since all these functions belong to $G$, and
  $\al{M}_{\al{A}}(\beta_1,\dots,\beta_k)$ is reflexive, we conclude that
  $\al{B}$ is a subalgebra of $\al{M}_{\al{A}}(\beta_1,\dots,\beta_k)$.
  By applying ${}^\frakC$ we get the inclusion $\supseteq$ in \eqref{eq-matrix-algs}.
  This completes the proof of \eqref{eq-matrix-algs}.
  \renewcommand{\qedsymbol}{$\diamond$}
  \end{proof}

  Now we are ready to prove \eqref{eq-HC}.
  Since our assumption $\beta_1,\dots,\beta_k\in I(0,\alpha)$
  implies that $[\beta_1,\dots,\beta_k]\le\beta_k\le\alpha$,
  the commutator $[\beta_1,\dots,\beta_k]$ is the least 
  congruence $\gamma$ of $\al{A}$ in the interval $I(0,\alpha)$
  which satisfies condition $(*)$.
  
  So, let $\gamma$ be a congruence of $\al{A}$ with $\gamma\le\alpha$.
  Condition $(*)$ in the definition of $[\beta_1,\dots,\beta_k]$
  can be expressed by relational clone operations as follows:
  \[
  \al{M}_{\al{A}}(\beta_1,\dots,\beta_k)\cap
  (\gamma\times_\alpha\dots\times_\alpha\gamma\times_\alpha A^2[\alpha])
  \subseteq \gamma\times_\alpha\dots\times_\alpha\gamma\times_\alpha\gamma
  \quad\text{(with $2^{k-1}$ factors)}
  \]
  provided the coordinates in $\mathbf{2}^k$ are listed so that
  $\epsilon0$ and $\epsilon1$ are consecutive for every
  $\epsilon\in\mathbf{2}^{k-1}$, and the last two coordinates are
  $1\dots10$ and $1\dots11$.
  All three compatible relations involved in this condition
  are reflexive and are contained in $A^{\mathbf{2}^k}[\alpha]$.
  Therefore, by Corollary~\ref{cor-A-n-alpha}(2),
  condition $(*)$ holds for $\gamma$ and
  $\al{M}_{\al{A}}(\beta_1,\dots,\beta_k)$ if and only if the analogous condition
  holds for $\gamma^\frakC$ and
  $\bigl(\al{M}_{\al{A}}(\beta_1,\dots,\beta_k)\bigr)^\frakC
  =\al{M}_{\al{C}}(\beta_1^\frakC,\dots,\beta_k^\frakC)$.
  Consequently, $\gamma\in I(0,\alpha)$ is the least congruence of $\al{A}$
  satisfying $(*)$ for $\al{M}_{\al{A}}(\beta_1,\dots,\beta_k)$
  if and only if $\gamma^\frakC$ is the least congruence of
  $\al{C}$ satisfying the analogous condition for
  $\al{M}_{\al{C}}(\beta_1^\frakC,\dots,\beta_k^\frakC)$.
  By the definition of the $k$-ary commutator, and by our remark
  in the preceding paragraph, this proves the desired equality \eqref{eq-HC}.
\end{proof}

For any algebra $\al{A}$ and congruence $\beta$ of $\al{A}$,
there is a largest congruence $\rho$ of $\al{A}$ such that
$[\rho,\beta]=0$.
This congruence is called the
\emph{centralizer of $\beta$}, and is denoted by $(0:\beta)$.
Theorem~\ref{thm-comm}, combined with this definition, immediately implies the
following fact.

\begin{cor}
  \label{cor-centr}
Let $\al{A}$ be an $\mathcal{F}$-algebra
with an onto homomorphism $\chi\colon\al{A}\to\al{I}$, and let
$\al{C}:=\frak{C}(\al{A},\chi)$,
$\alpha:=\ker(\chi)$.
The isomorphism ${}^\frakC$ between the interval $I(0,\alpha)$ of
the congruence lattice of $\al{A}$ and the congruence lattice of
$\al{C}$ (see Corollary~\ref{cor-congr}(1))
preserves centralizers in the following sense:
for any congruence $\beta\in I(0,\alpha)$ of $\al{A}$
we have that
\[
((0:\beta) \wedge\alpha)^\frakC=(0:\beta^\frakC).
\]
\end{cor}

\subsection{Varieties, clones, and Maltsev conditions}
In this subsection let $\var{V}$ be any variety of (nonempty)
$\lngF$-algebras, and let 
$\al{I}=([m],\lngF)$ ($m>0$) be a finite algebra in $\var{V}$.
We will use the notation $\Vcomma$ for the full subcategory of
$\comma$
consisting of those objects $(\al{A},\chi)$ for which
$\al{A}\in\var{V}$.

The discussions following Proposition~\ref{prp-functorP} and
Corollary~\ref{cor-functorC} imply that if $\var{V}$ is the variety
of all (nonempty) $\lngF$-algebras, then the class of all isomorphic copies
of $\hat{\lngF}_{\al{I}}$-algebras in the range
$\frak{C}\Vcomma=\frak{C}\comma=\DAlg^\boxplus(\hat{\lngF_{\al{I}}})$
of the functor
$\frak{C}\colon\comma\to\DAlg^\boxplus(\hat{\lngF_{\al{I}}})$
is the variety $\mathcal{D}^+(\hat{\lngF_{\al{I}}})$  of $\hat{\lngF_{\al{I}}}$-algebras.
We can combine earlier results of this subsection to obtain
an analogous conclusion for all subvarieties $\var{V}$ of the variety
of all $\lngF$-algebras.

\begin{cor}
  \label{cor-var}
If $\var{V}$ is a variety of $\lngF$-algebras and   
$\al{I}=([m],\lngF)$ $(m>0)$ is a finite algebra in $\var{V}$, then
the class
\[
\var{V}^\frakC:=\mathbb{I}\,\frak{C}\Vcomma
\]
of all isomorphic copies of algebras in
$\frak{C}\Vcomma$ is a variety
of $\hat{\lngF}_{\al{I}}$-algebras.
\end{cor}  

\begin{proof}
  The fact that $\var{V}^\frakC$ is closed under taking products and subalgebras
  follows from the discussion preceding Theorem~\ref{thm-subprod} and from
  the first paragraph of the proof of that theorem.
  To show that $\var{V}^\frakC$ is closed under taking quotients, consider
  an algebra $\al{C}$ in $\var{V}^\frakC$ and a congruence $\gamma$ of $\al{C}$.
  Since $\al{C}$ is isomorphic to an algebra of the form
  $\frak{C}(\al{A},\chi)$ for some $\al{A}\in\var{V}$ and some onto homomorphism
  $\chi\colon\al{A}\to\al{I}$, we may assume without loss of generality that
  $\al{C}$ is actually equal to $\frak{C}(\al{A},\chi)$.
  By Corollary~\ref{cor-congr}, $\gamma=\beta^\frakC$ for a unique congruence
  $\beta$ of $\al{A}$ such that $\beta\le\ker(\chi)$.
  Now Corollary~\ref{cor-quotient} shows that
  $\al{C}/\beta^\frakC\cong\frak{C}(\al{A}/\beta,\chi/\beta)\in\var{V}^\frakC$,
  completing the proof.
\end{proof}  

  In the next theorem we exhibit a
relationship between free algebras
$\al{F}_{\var{V}^\frakC}(X)$ in the variety $\var{V}^\frakC$ and
free algebras in $\var{V}$.
The following notation will be useful.
For any set $X$ we define $\wec{X}^\flat$ to be the set
$X\times[m]$, but we will use the notation $x^{(i)}$ for its element $(x,i)$
for every $x\in X$ and $i\in[m]$. 
Furthermore, we define
\begin{align*}
\wec{x}&:=\begin{bmatrix} x\us{1}\\ \vdots\\ x\us{m} \end{bmatrix}
\ \text{for each $x\in X$, and}\\
  \wec{X}&:=\{\wec{x}:x\in X\}.
\end{align*}

\begin{thm} \label{thm-free}
  Let $\var{V}$ be a variety of $\lngF$-algebras, let
  $\al{I}=([m];\mathcal{F})$ $(m>0)$ be a finite algebra in $\var{V}$,
  and let $\var{V}^\frakC$ be the corresponding
  variety of $\hat{\lngF}_{\al{I}}$-algebras
  defined in Corollary~\ref{cor-var}.
    For any non-empty set $X$ we have that
    \begin{equation}
      \label{eq-free1}
      \al{F}_{\var{V}^\frakC}(X) \cong
      \frak{C}\bigl( \al{F}_{\var{V}}(\wec{X}^\flat),\xi \bigr)
    \end{equation}
    where $\xi$ is the unique homomorphism
    $\al{F}_{\var{V}}(\wec{X}^\flat) \to \al{I}$ in $\var{V}$ that extends
    the set map
    $\wec{X}^\flat\to[m],\ x^{(i)} \mapsto i$.
\end{thm}

\begin{proof}
  Let $\al{F} := \frak{C}\bigl( \al{F}_{\var{V}}(\wec{X}^\flat),\xi \bigr)$.
  It follows from the definitions of $\frak{C}$ and $\wec{X}$ that
  $\wec{X}\subseteq F$. Since $X\to\wec{X}$, $x\mapsto\wec{x}$,
  is a bijection,
  to prove \eqref{eq-free1} it suffices to show that $\al{F}$ is a
  free algebra in $\var{V}^\frakC$ with free generating set $\wec{X}$.

 To this end let $\al{C}\in\var{V}^\frakC$. Then $\al{C}$ is isomorphic to $\frak{C}(\al{A},\chi)$ for some $\al{A}$ in
 $\var{V}$ and some onto homomorphism
 $\chi\colon \al{A}\to\al{I}$.
 Without loss of generality assume $\al{C}=\frak{C}(\al{A},\chi)$, and
 consider an arbitrary set map
 $G\colon \wec{X} \to C = \prod_{i\in [m]} \chi^{-1}(i)$.
 Our goal is to show that $G$ extends to a unique homomorphism
 $\al{F}\to\al{C}$.
 
 Note first that
 there exists a (unique) set map $g\colon\wec{X}^\flat\to A$ such that
 \[
 G(\wec{x}) =
 \begin{bmatrix} g(x\us{1}) \\ \vdots \\ g(x\us{m}) \end{bmatrix}
\text{ and $g(x\us{i})\in\chi^{-1}(i)$}
 \text{ for all } x\in X \text{ and } i\in [m].
 \]
 Thus, $g\colon \wec{X}^\flat \to A$ extends uniquely to a homomorphism
 $\psi\colon\al{F}_{\var{V}}(\wec{X}^\flat)\to\al{A}$. Moreover, $\psi$ satisfies
 the equality $\chi\circ \psi = \xi$, because it follows from our
 definitions that
 \[ \chi(\psi(x\us{i})) = \chi(g(x\us{i})) = i = \xi(x\us{i})
 \text{ for all } x\us{i} \in \wec{X}^\flat. \]
 Hence $\psi$ is a morphism
   $\bigl(\al{F}_{\var{V}}(\wec{X}^\flat),\xi\bigr)\to(\al{A},\chi)$
 in $\Vcomma$.
 Thus $\frak{C}(\psi)$ is a homomorphism $\al{F}\to\al{C}$
 by Corollary~\ref{cor-functorC}.
  By its construction, $\frak{C}(\psi)$ extends $G$.

 To show that this extension is unique,
 consider any homomorphism $\Phi\colon \al{F}\to\al{C}$
 that extends $G$. By Corollary~\ref{cor-functorC},
 $\Phi=\frak{C}(\phi)$
   for some morphism 
   $\phi\colon\bigl(\al{F}_{\var{V}}(\wec{X}^\flat),\xi\bigr)\to(\al{A},\chi)$
   in $\Vcomma$.
 It follows that $\phi$ extends $g$.
 Hence $\phi = \psi$ since $\al{F}_{\var{V}}(\wec{X}^\flat)$ is free in
 $\var{V}$ with free generating set $\wec{X}^\flat$.
   Thus $\Phi={}\frak{C}(\phi)=\frak{C}(\psi)$.
 This concludes the proof of the theorem.
\end{proof}  

  Recall from \cite{taylor}
  that the \emph{clone of a variety $\var{W}$} is the multisorted
  algebra with sorts indexed by positive integers $k$, the $k$-th sort
  being the set of all \emph{$k$-ary terms of $\var{W}$}
  (i.e., all terms in the language of $\var{W}$
  in the first $k$ variables $v_1,\dots,v_k$,
  modulo the identities true in $\var{W}$),
  and the operations are superpositions of terms and nullary operations
  naming the variables in each sort.
  Since for every positive integer $k$,
  the $k$-ary terms of $\var{W}$
  are in one-to-one correspondence with the elements of the free algebra
  $\al{F}_{\var{W}}(X)$ for a $k$-element set $X$,
  the isomorphism \eqref{eq-free1} in Theorem~\ref{thm-free}
  yields a description for the clone of the variety
  $\var{V}^\frakC$ in terms of the clone of $\var{V}$.
    
  We state this description in Corollary~\ref{cor-terms}
  below.
  We will use the matrix notation introduced in our discussion following
  Theorem~\ref{prp-functorP}. Furthermore,
  $\wec{e}_{m\times k}$ will denote
  the $m\times k$ matrix in which the $i$-th row is $[i\ i\ . . .\ i]$
  for every $i\in[m]$.

\begin{cor}
  \label{cor-terms}
Let $\var{V}$, $\al{I}$, and $\var{V}^\frakC$ be the same as in
Theorem~\ref{thm-free}.
For any positive integer $k$,
there is a bijection
$\displaystyle
  T(\al{x})\mapsto
  \begin{bmatrix}
    t\us{1}(\al{x})\\
    \vdots\\
    t\us{m}(\al{x})
  \end{bmatrix}
$  
between
\begin{itemize}
\item
  the $k$-ary terms $T$ of $\var{V}^\frakC$ and
\item  
the $m$-tuples $t\us{1},\dots,t\us{m}$ 
of $mk$-ary terms of $\var{V}$ satisfying
\begin{equation}
  \label{eq-term-restr}
    \text{$t\us{i}(\wec{e}_{m\times k})=i$ in $\al{I}$ for all $i\in[m]$}
\end{equation}
\end{itemize}
so that the following equality holds in the algebra
    $\al{C}:=\frak{C}(\al{A},\chi)$ for
    each $\al{A}\in\var{V}$ with an onto homomorphism
    $\chi\colon\al{A}\to\al{I}$:
  \begin{equation}
    \label{eq-term-eq}
    \qquad
  T(\al{a})=
  \begin{bmatrix}
    t\us{1}(\al{a})\\
    \vdots\\
    t\us{m}(\al{a})
  \end{bmatrix}
  \quad
  \text{for every $m\times k$ matrix
    $\al{a}\in\Bigl(\prod_{i\in[m]}\chi^{-1}(i)\Bigr)^k=C^k$,}
  \end{equation}
  where $T$ is applied to the $k$ columns of $\al{a}$, 
  and $t\us{1},\dots,t\us{m}$ are applied to the $mk$ entries of $\al{a}$. 
\end{cor}

For every $k$-ary term $T$ of $\var{V}^\frakC$,
the $mk$-ary terms $t\us{i}$ ($i\in[m]$) of $\var{V}$ 
satisfying conditions
\eqref{eq-term-restr}--\eqref{eq-term-eq} above
will be referred to as \emph{the coordinate terms of $T$}.

\begin{cor}
  \label{cor-clone-fingen}
Let $\var{V}$, $\al{I}$, and $\var{V}^\frakC$ be the same as in
Theorem~\ref{thm-free}.
If the clone of $\var{V}$ is finitely generated, then so is the clone
of $\var{V}^\frakC$.
\end{cor}

\begin{proof}
If the clone of $\var{V}$ is finitely generated, then
$\var{V}$ is equivalent to its reduct $\var{V}^\circ$
to a finite sublanguage $\lngF^\circ$ of $\lngF$.
Clearly, the reduct $\al{I}^\circ=([m];\lngF^\circ)$ of $\al{I}$
belongs to $\var{V}^\circ$, and the variety
$(\var{V}^\circ)^\frakC=
\mathbb{I}\,\bigl(\var{V}^\circ\,{\downdownarrows}\,\al{I}^\circ\bigr)$
has a finite language, $\widehat{\lngF^\circ}_{\al{I}^\circ}$.
Since $\var{V}$ and $\var{V}^\circ$ are equivalent varieties, i.e., they have
the same terms, and any $m$-tuple $t\us{1},\dots,t\us{m}$ of $mk$-ary terms
among them satisfies condition \eqref{eq-term-restr}
for $\al{I}$ if and only if it does so
for $\al{I}^\circ$, Corollary~\ref{cor-terms} shows that
$\var{V}^\frakC$ and $(\var{V}^\circ)^\frakC$ are equivalent varieties.
As the latter variety has a finite language, their clones are
finitely generated.
\end{proof}  

  Corollaries~\ref{cor-terms}--\ref{cor-clone-fingen} carry over easily
  from varieties to single algebras.

  \begin{cor}
    \label{cor-termops}
  Let $(\al{A},\chi)$ be an object of $\comma$, and
  let $\al{C}:=\frak{C}(\al{A},\chi)$.
  \begin{enumerate}
  \item[{\rm(1)}]
  A function
  $T\colon C^k\to C$ (where $C=\prod_{i\in[m]}\chi^{-1}(i)$)
  is a $k$-ary term operation of $\al{C}$
  if and only if $\al{A}$ has $mk$-ary term operations
  $t\us{1},\dots,t\us{m}$ satisfying condition \eqref{eq-term-restr}
  such that \eqref{eq-term-eq} holds.
  \item[{\rm(2)}]
  If the clone of $\al{A}$ is finitely generated,
  then so is the clone of $\al{C}$.
  \end{enumerate}
  \end{cor}

Corollary~\ref{cor-terms}
also allows us to transfer idempotent terms from $\var{V}$
to $\var{V}^\frakC$.

\begin{cor}
  \label{cor-clone-idem}
 Let $\var{V}$, $\al{I}$, and $\var{V}^\frakC$ be the same as in
Theorem~\ref{thm-free}.
  \begin{enumerate}
  \item[{\rm(1)}]
    The map
    \[
    t(x_1,\dots,x_k)\mapsto
    \breve{t}(\wec{x}_1,\dots,\wec{x}_k):=
    \begin{bmatrix}
      t(x\lsus{1}{1},\dots,x\lsus{k}{1})\\
      \vdots\\
      t(x\lsus{1}{m},\dots,x\lsus{k}{m})
      \end{bmatrix},
    \]
    which assigns to every $k$-ary idempotent term
    $t(x_1,\dots,x_k)$ of $\var{V}$
    the $k$-ary idempotent term $\breve{t}(\wec{x}_1,\dots,\wec{x}_k)$ of
    $\var{V}^\frakC$ with coordinate terms $t(x\lsus{1}{i},\dots,x\lsus{k}{i})$
    $(i\in[m])$,
    is a clone homomorphism of the clone of
    all idempotent terms of $\var{V}$ into the clone of
    all idempotent terms of $\var{V}^\frakC$.
  \item[{\rm(2)}]
    Hence, $\var{V}^\frakC$
    satisfies every idempotent Maltsev condition
    that holds in $\var{V}$.
  \end{enumerate}
\end{cor}  

\begin{proof}
  Since $t$ is an idempotent term, i.e., the identity $t(x,\dots,x)\approx x$
  holds in $\var{V}$, the coordinate terms $t(x\lsus{1}{i},\dots,x\lsus{k}{i})$
  satisfy $t(i,\dots,i) = i$ for every $i\in [m]$. Thus 
    the coordinate terms $t(x\lsus{1}{i},\dots,x\lsus{k}{i})$ satisfy
    condition \eqref{eq-term-restr}. It follows from
    Corollary~\ref{cor-terms} that
    $\var{V}^\frakC$ has a term $\breve{t}$ with these terms as coordinate terms.
   
  Clearly, $\breve{t}$ is an idempotent term of $\var{V}^\frakC$. It is also
  straightforward to check that the map $t\mapsto\breve{t}$ is a clone
  homomorphism. This proves statement (1).

  For (2), recall that a strong Maltsev condition
  is a primitive positive sentence in the language of clones, while a
  Maltsev condition is a disjunction of a weakening infinite sequence
  of strong Maltsev conditions.
  Since the satisfaction of conditions of this form
  is preserved under homomorphisms,
  we get from part (1) that $\var{V}^\frakC$ satisfies every idempotent
  Maltsev condition that holds in $\var{V}$.
\end{proof}  

We conclude this subsection by recording some finiteness properties
of varieties preserved by $\frak{C}$.

\begin{cor}
  \label{cor-presentations}
Let $\var{V}$, $\al{I}$, and $\var{V}^\frakC$ be the same as in
  Theorem~\ref{thm-free}.
\begin{enumerate}
\item[{\rm(1)}]
  $\var{V}$ is locally finite if and only if $\var{V}^\frakC$ is locally finite.
\item[{\rm(2)}]  
  For every algebra $\al{A}\in\var{V}$ with an onto homomorphism
  $\chi\colon\al{A}\to\al{I}$ we have that
  $\al{A}$ is finitely presented in $\var{V}$ if and only if the algebra
  $\frak{C}(\al{A},\chi)$ is finitely presented in
  $\var{V}^\frakC$.
\end{enumerate}

\end{cor}  

\begin{proof}
  We will use the notation introduced in the paragraph
    preceding Theorem~\ref{thm-free}.

    Statement~(1) will follow if we prove that for every finite
    set $X$,
    \begin{align*}
      |\al{F}_{\var{V}^\frakC}(X)| & \le |\al{F}_{\var{V}}(\wec{X}^\flat)|^m,
      \quad\text{and}\\
      |\al{F}_{\var{V}}(\wec{X}^\flat)| & \le m|\al{F}_{\var{V}^\frakC}(X)|,
    \end{align*}
    where
    $|\wec{X}^\flat|=m|X|$. The first inequality is an immediate consequence
    of the isomorphism in \eqref{eq-free1}. For the second equality,
    combine \eqref{eq-free1} with the observation that
    every element of $\al{F}_{\var{V}}(\wec{X}^\flat)$
    belongs to $\xi^{-1}(i)$ for some $i\in[m]$.
    
    To prove (2), let $\al{A}$ be an algebra in $\var{V}$ with an
    onto homomorphism $\chi\colon\al{A}\to\al{I}$. By the definition of
    $\var{V}^\frakC$ (see Corollary~\ref{cor-var}), we have
    $\frak{C}(\al{A},\chi)\in\var{V}^\frakC$.
    Assume first that $\frak{C}(\al{A},\chi)$
    is finitely presented in $\var{V}^\frakC$, that is, for some
    finite set $X$, $\frak{C}(\al{A},\chi)$
    is a quotient of $\al{F}_{\var{V}^\frakC}(X)$ by a
    finitely generated congruence. 
    By Theorem~\ref{thm-free} and Corollary~\ref{cor-congr}, this is
    equivalent to the condition that
\[
\frak{C}(\al{A},\chi)
\cong \frak{C}\bigl(\al{F}_{\var{V}}(\wec{X}^\flat),\xi\bigr)/\beta^\frakC
\]
for some finite set $X$ and some finitely generated congruence
$\beta\leq \ker(\xi)$ of $\al{F}_{\var{V}}(\wec{X}^\flat)$.
By Corollary~\ref{cor-quotient}, the last displayed isomorphism can be
rewritten as follows:  
\[
\frak{C}(\al{A},\chi) \cong
\frak{C}\bigl(\al{F}_{\var{V}}(\wec{X}^\flat)/\beta, \xi/\beta\bigr).
\]
Since $\frak{C}$ is an isomorphism between the category $\Vcomma$ and
its image in $\var{V}^\frakC$ (see
Corollaries~\ref{cor-functorC} and \ref{cor-var}),
it follows that $\al{F}_{\var{V}}(\wec{X}^\flat)/\beta \cong \al{A}$.
Hence, $\al{A}$ is finitely presented in $\var{V}$.

  Conversely assume that $\al{A}\in\var{V}$ is finitely presented in $\var{V}$.
  That is, we have a finite set $X$ and an onto homomorphism
  $h\colon \al{F}_{\var{V}}(X) \to \al{A}$ such that $\ker(h)$ is
  a finitely generated congruence of $\al{F}_{\var{V}}(X)$.
  For each $i\in[m]$ fix some
element $d\us{i}\in\chi^{-1}(i)\,(\subseteq A)$.
Let $h'$ denote the unique homomorphism
$h'\colon \al{F}_{\var{V}}(\wec{X}^\flat) \to \al{A}$ that extends the set
map $\wec{X}^\flat\to A$ defined by
\[
h'(x\us{i}) = \begin{cases} h(x) & \text{if } h(x)\in\chi^{-1}(i), \\
  d\us{i} & \text{otherwise.} \end{cases}
\]
  Since $h=h'\circ\iota$ holds for the injection
  $\iota\colon X\to\wec{X}^\flat$, $x\mapsto x^{\chi(h(x))}$,
  it follows that, up to the isomorphism
  of $\al{F}_{\var{V}}(X)$ and $\al{F}_{\var{V}}\bigl(\iota(X)\bigr)$
  induced by $\iota$, $h$ and the restriction of $h'$ to
  $\al{F}_{\var{V}}\bigl(\iota(X)\bigr)$ are the same homomorphisms onto
  $\al{A}$.
  Hence, $\al{F}_{\var{V}}(X)/\ker(h)\cong\al{A}\cong
    \al{F}_{\var{V}}(\wec{X}^\flat)/\ker(h')$.
    Since $\ker(h)$ is finitely generated and
    $\wec{X}^\flat$ is finite, we get that
  $\beta := \ker(h')$
  is a finitely generated congruence of $\al{F}_{\var{V}}(\wec{X}^\flat)$
  and $\al{F}_{\var{V}}(\wec{X}^\flat)/\beta\cong\al{A}$ is finitely presented.

 For $\xi := \chi\circ h'$, Theorem~\ref{thm-free} implies that
$\al{F}_{\var{V}^\frakC}(X)\cong\frak{C}\bigl(\al{F}_{\var{V}}(\wec{X}^\flat),\xi\bigr)$
is a finitely generated free algebra in $\var{V}^\frakC$.
Since $\beta\leq\ker(\xi)$, Corollary~\ref{cor-congr} yields that
$\beta^\frakC$ is a finitely generated congruence
of $\frak{C}\bigl(\al{F}_{\var{V}}(\wec{X}^\flat),\xi\bigr)$.
Now, Corollary~\ref{cor-quotient} yields the leftmost $\cong$ below 
\[
\frak{C}\bigl(\al{F}_{\var{V}}(\wec{X}^\flat), \xi\bigr)/\beta^\frakC
\cong\frak{C}\bigl(\al{F}_{\var{V}}(\wec{X}^\flat)/\beta, \xi/\beta\bigr)
\cong \frak{C}(\al{A},\chi),
\]
while the rightmost $\cong$ follows from the fact that
$\bigl(\al{F}_{\var{V}}(\wec{X}^\flat)/\beta, \xi/\beta\bigr)$
and $(\al{A},\chi)$ are isomorphic objects in $\Vcomma$, mediated by
the onto homomorphism $h'$ with kernel $\beta$.
This proves that $\frak{C}(\al{A},\chi)$ is a finitely presented algebra in
$\var{V}^\frakC$.
\end{proof}

\subsection{Polynomial operations and TCT types}

For any object $(\al{A},\chi)$ of $\comma$, we have 
a description for the term operations
of the algebra $\frak{C}(\al{A},\chi)$ via 
term operations of $\al{A}$ in Corollary~\ref{cor-termops}(1).
The next corollary presents an analogous description
for the polynomial operations of $\frak{C}(\al{A},\chi)$.

\begin{cor}
  \label{cor-poly}
  Let $(\al{A},\chi)$ be an object of $\comma$, and
  let $\al{C}:=\frak{C}(\al{A},\chi)$.
  A function
  $P\colon C^k\to C$ (where $C=\prod_{i\in[m]}\chi^{-1}(i)$)
  is a $k$-ary polynomial operation of $\al{C}$
  if and only if $\al{A}$ has $mk$-ary polynomial operations
  $p\us{1},\dots,p\us{m}$ such that for every $m\times k$ matrix
  $\al{a}\in\bigl(\prod_{i\in[m]} \chi^{-1}(i)\bigr)^k=C^k$,
  \begin{equation}
    \label{eq-poly}
  p\us{i}(\al{a})\in \chi^{-1}(i)\ \ \text{for every $i\in[m]$,\quad
    and\quad }
    P(\al{a})=
  \begin{bmatrix}
    p\us{1}(\al{a})\\
    \vdots\\
    p\us{m}(\al{a})
  \end{bmatrix}.
  \end{equation}
\end{cor}

\begin{proof}
  We want to deduce this statement from
    Corollary~\ref{cor-termops}(1)
    by using the fact that for every
  algebra $\al{U}$ the polynomial operations
  of $\al{U}$ are the term operations of the constant expansion of $\al{U}$.
  In keeping with our convention of excluding nullary symbols, we define the
  \emph{constant expansion} of an algebra
  $\al{U}=(U;\mathcal{L})$ to be the algebra
  $\al{U}^\ccc=(U;\mathcal{L}\cup\{\ccc_u:u\in U\})$
  where each $\ccc_u$ is a unary symbol, and is interpreted as the unary
  constant operation on $U$ with value $u$.

  Now let $(\al{A},\chi)$ satisfy the assumptions of Corollary~\ref{cor-poly},
  and let $\al{A}^\ccc$ be the constant expansion of $\al{A}$.
  The language of $\al{A}^\ccc$ is $\lngF^\ccc=\lngF\cup\{\ccc_a:a\in A\}$.
  Since $\chi$ is an onto homomorphism $\al{A}\to\al{I}$, we can
  expand $\al{I}$ to get an $\lngF^\ccc$-algebra
  $\al{I}^\ccc=([m],\lngF^\ccc)$ by interpreting each symbol $\ccc_a$ ($a\in A$)
  as the unary constant function with value $\chi(a)$ on $[m]$.
  Thus, $\chi$ is an onto homomorphism $\al{A}^\ccc\to\al{I}^\ccc$, and
  $(\al{A}^\ccc,\chi)$ is an object of the category
  $\bigl(\Alg(\lngF^\ccc)\,{\downdownarrows}\,\al{I}^\ccc\bigr)$.
  By applying the functors $\frak{C}$ with domain categories
  $\comma$ and $\bigl(\Alg(\lngF^\ccc)\,{\downdownarrows}\,\al{I}^\ccc\bigr)$
  to the objects $(\al{A},\chi)$ and $(\al{A}^\ccc,\chi)$, respectively,
  we get the $\hat{\lngF}_{\al{I}}$-algebra $\frak{C}(\al{A},\chi)$
  and the $\hat{\lngF^\ccc}_{\al{I}^\ccc}$-algebra
  $\frak{C}(\al{A}^\ccc,\chi)$.

  \begin{clm}
    \label{clm-poly}
    The clone of term operations of $\frak{C}(\al{A}^\ccc,\chi)$ coincides
    with the clone of term operations of the constant expansion of
    $\frak{C}(\al{A},\chi)$.
  \end{clm}

  \begin{proof}[Proof of Claim~\ref{clm-poly}]
    The set $C:=\prod_{i\in[m]}\chi^{-1}(i)$ is the universe of all
    three algebras that play a role in the claim:
    $\frak{C}(\al{A},\chi)$, $\frak{C}(\al{A}^\ccc,\chi)$, and
    the constant expansion of $\frak{C}(\al{A},\chi)$.
    It also follows from the definition of the functor $\frak{C}$ that
    $\frak{C}(\al{A}^\ccc,\chi)$ is obtained from  
    $\frak{C}(\al{A},\chi)$ by adding
    the unary operations $\hatt{\ccc_a}_i$ ($a\in A$, $i\in[m]$),
    while the constant expansion of $\frak{C}(\al{A},\chi)$
    is obtained from  $\frak{C}(\al{A},\chi)$ by adding
    the unary constant operations $\ccc_{\wec{c}}$ ($\wec{c}\in C$).
    Therefore, to prove the claim it suffices to show that
    \begin{enumerate}
    \item[(1)]
      every operation $\hatt{\ccc_a}_i$ ($a\in A$, $i\in[m]$)
      is a term operation of the constant expansion of
      $\frak{C}(\al{A},\chi)$, and
    \item[(2)]
      every unary constant operation $\ccc_{\wec{c}}$ ($\wec{c}\in C$) is a
      term operation of $\frak{C}(\al{A}^\ccc,\chi)$.
    \end{enumerate}
    (1) follows by observing that if $a\in \chi^{-1}(i')$ 
    and $\wec{c}$ is an element of $C=\prod_{\ell\in[m]}\chi^{-1}(\ell)$
    with $i'$-th coordinate $a$, then for each $i\in[m]$ we have that
    \[
    \hatt{\ccc_a}_i(\wec{x})=
  \diag\bigl(\wec{x},\dots,\wec{x},\underbrace{\ccc_{\wec{c}}(\wec{x})}_{\text{$i'$-th}}
    \wec{x},\dots,\wec{x}\bigr).
    \]
    For (2), notice that if
    $\wec{c}=(c_1,\dots,c_m)\in\prod_{i\in[m]}\chi^{-1}(i)=C$,
    then
    the unary operation $\ccc_{\wec{c}}$ is the composition (in any order)
    of the unary operations $\hatt{\ccc_{c_1}}_1,\dots,\hatt{\ccc_{c_m}}_m$.
  \renewcommand{\qedsymbol}{$\diamond$}
  \end{proof}  

To summarize, we have that the clone of polynomial operations of
$\frak{C}(\al{A},\chi)$ is the clone of term operations of
$\frak{C}(\al{A}^\ccc,\chi)$, while the clone of polynomial operations
of $\al{A}$ is the clone of term operations of $\al{A}^\ccc$.
Therefore, if we apply Corollary~\ref{cor-termops}(1)
to the object $(\al{A}^\ccc,\chi)$ of the category
$\bigl(\Alg(\lngF^\ccc)\,{\downdownarrows}\,\al{I}^\ccc\bigr)$,
we get the conclusion of Corollary~\ref{cor-poly}.
\end{proof}
  
For the basic concepts and notation of tame congruence theory we refer
the reader to \cite{hobby-mckenzie}. If for a covering pair
$\delta\prec\theta$ of congruences in some finite algebra $\al{A}$ we have that
$\typ_{\al{A}}(\delta,\theta)={\bf2}$, and hence the minimal algebras
$\al{A}|_N/\delta|_N$ associated to all
$\langle\delta,\theta\rangle$-traces $N$ are weakly isomorphic 
one-dimensional vector spaces over a finite field,
we will refer to the characteristic of this field as
\emph{the characteristic of $\langle\delta,\theta\rangle$}.
In particular, if $\theta$ is an atom in the congruence lattice of $\al{A}$,
then the characteristic of the prime quotient
$\langle 0,\theta\rangle$ will also be called
\emph{the characteristic of $\theta$}.

\begin{thm}
  \label{thm-tct}
  Let $\mathcal{F}$ be an algebraic language, and let
  $\al{I}=([m];\mathcal{F})$ and $\al{A}$ be finite $\mathcal{F}$-algebras.
  If  $(\al{A},\chi)$ is an object of $\comma$ with $\alpha:=\ker(\chi)$,
  then for any covering pair $\delta\prec\theta$ of congruences of $\al{A}$
  in the interval $I(0,\alpha)$ and for every
  $\langle\delta,\theta\rangle$-trace $N$ of $\al{A}$ there exist a
  $\langle\delta^\frakC,\theta^\frakC\rangle$-trace $\check{N}$ of
  $\al{C}:=\frak{C}(\al{A},\chi)$ and a weak isomorphism
  $\al{A}|_N\to\al{C}|_{\check{N}}$ between the induced algebras which
  maps $\delta|_N$ to $\delta^\frakC|_{\check{N}}$.
  Consequently,
    \[
  \typ_{\al{A}}(\delta,\theta)=\typ_{\al{C}}(\delta^\frakC,\theta^\frakC);
  \]
  moreover,
  if $\typ_{\al{A}}(\delta,\theta)=\typ_{\al{C}}(\delta^\frakC,\theta^\frakC)={\bf2}$,
  then the prime quotients $\langle\delta,\theta\rangle$ and
  $\langle\delta^\frakC,\theta^\frakC\rangle$ have the same prime characteristic.
\end{thm}  

\begin{proof}
We will use the notation $D\us{i}:=\chi^{-1}(i)$ ($i\in[m]$)
for the equivalence classes of $\alpha$.
So, $C:=\prod_{i\in[m]}D\us{i}$ is
the underlying set of the algebra $\al{C}=\frak{C}(\al{A},\chi)$.

Let $\delta\prec\theta\,(\le\alpha)$ be congruences in $\al{A}$, and
let $N$ be a $\langle\delta,\theta\rangle$-trace of $\al{A}$.
Then there exist a unary polynomial operation $e=e^2$
of $\al{A}$ such that $U:=e(A)$ is a
$\langle\delta,\theta\rangle$-minimal set of $\al{A}$, and
$N=U\cap (a/\theta)$ for some $a\in U$.
Furthermore,
$\theta|_N$ is the full relation on $N$ and $\delta|_N\subsetneq\theta|_N$.

Since $\theta\leq\alpha$, $N$ lies in a single $\alpha$-class;
without loss of generality we may assume that $N\subseteq D\us{1}$.
Now we choose and fix elements
$c\us{2}\in D\us{2},\dots,c\us{m}\in D\us{m}$ from the remaining
$\alpha$-classes. For $2\le i\le m$ let $q\us{i}$ denote the unary constant
polynomial operation of $\al{A}$ with value $c\us{i}$, while for $i=1$
let $q\us{1}:=e$. Notice that $e(D\us{1})\subseteq D\us{1}$, because
$D\us{1}=a/\alpha$ and $e(a)=a$ (as $a\in N\subseteq U=e(A)$).
Therefore, it follows from 
Corollary~\ref{cor-poly} that the function
$\check{e}\colon C\to C$ defined by
\[
\check{e}(\wec{a}):=\begin{bmatrix} q\us{1}(a\us{1})\\
                     q\us{2}(a\us{2})\\
                     \vdots\\
                     q\us{m}(a\us{m})
             \end{bmatrix}        
           =\begin{bmatrix} e(a\us{1})\\
                     c\us{2}\\
                     \vdots\\
                     c\us{m}
           \end{bmatrix}
           \quad
           \text{for all}\quad
           \wec{a}=\begin{bmatrix} a\us{1}\\
                     a\us{2}\\
                     \vdots\\
                     a\us{m}
             \end{bmatrix}\in C        
\]
is a polynomial operation of $\al{C}$.
Clearly, $\check{e}=\check{e}^2$, and for the set $\check{U}:=\check{e}(C)$
we have that
\[
\check{U}=e(D\us{1})\times\{c\us{2}\}\times\dots\times\{c\us{m}\}
=(U\cap D\us{1})\times\{c\us{2}\}\times\dots\times\{c\us{m}\}.
\]
For each $u\in U\cap D\us{1}$ let
$\check{u}:=(u,c\us{2},\dots,c\us{m})$; thus, we have
a bijection $U\cap D\us{1}\to\check{U}$, $u\mapsto\check{u}$.
Restricting this mapping to the set $N\,(\subseteq U\cap D\us{1})$ we obtain a
bijection
\[
N\to\check{N}:=N\times\{c\us{2}\}\times\dots\times\{c\us{m}\},
\quad u\mapsto\check{u}.
\]
It follows from Corollary~\ref{cor-congr}(1) that 
$\delta^\frakC|_{\check{N}}\subsetneq\theta^\frakC|_{\check{N}}$ and
$\check{N}=\check{U}\cap (\check{a}/\theta^\frakC)$ (so $\theta^\frakC|_{\check{N}}$
is the full relation on $\check{N}$).
To establish that $\check{N}$ is a $\langle\delta^\frakC,\theta^\frakC\rangle$-trace
of $\al{C}$ it remains to show that $\check{U}$ is a
$\langle\delta^\frakC,\theta^\frakC\rangle$-minimal set of $\al{C}$.

Assume not, and let $\check{V}:=V\times\{c\us{2}\}\times\dots\times\{c\us{m}\}$
be a proper subset of $\check{U}$ which is
$\langle\delta^\frakC,\theta^\frakC\rangle$-minimal in $\al{C}$.
Then there exists a unary polynomial $g$ of $\al{C}$ such that
\begin{equation}
  \label{eq-props-g}
  \check{V}=g(C),\quad g=g^2,\quad \text{and}\quad
  \delta^\frakC|_{\check{V}}\subsetneq\theta^\frakC|_{\check{V}}.
\end{equation}
  Hence, by Corollary~\ref{cor-congr}(1),
$\delta|_{V}\subsetneq\theta|_{V}$ where $V\subsetneq U\cap D\us{1}$.
Since both $g$ and $\check{e}$ fix every element of
$V\,(\subseteq\check{U})$, so will the unary polynomial $g\circ\check{e}$.
Therefore, by replacing $g$ with $g\circ\check{e}$,
we may assume from now on that, in addition to \eqref{eq-props-g},
$g$ also satisfies $g=g\circ\check{e}$.
Let $g\us{1}$ denote the $m$-ary polynomial operation of $\al{A}$
that is the first coordinate function of $g$ according to
Corollary~\ref{cor-poly}, and define a unary polynomial operation
$f$ of $\al{A}$ by 
$f(x):=g\us{1}\bigl(e(x),c\us{2},\dots,c\us{m}\bigr)$.
Then,
\begin{multline*}
g(\wec{a}) 
=(g\circ\check{e})(\wec{a})
=g\left(
\begin{bmatrix} e(a\us{1})\\
                     c\us{2}\\
                     \vdots\\
                     c\us{m}
\end{bmatrix}
\right)
=
\begin{bmatrix} g\us{1}\bigl(e(a\us{1}),c\us{2},\dots,c\us{m}\bigr)\\
                     c\us{2}\\
                     \vdots\\
                     c\us{m}
\end{bmatrix}
=
\begin{bmatrix} f(a\us{1})\\
                     c\us{2}\\
                     \vdots\\
                     c\us{m}
\end{bmatrix}\\
\text{for all}\quad
           \wec{a}=\begin{bmatrix} a\us{1}\\
                     a\us{2}\\
                     \vdots\\
                     a\us{m}
             \end{bmatrix}\in C.        
\end{multline*}
The first property of $g$ in \eqref{eq-props-g} implies that
$f(U\cap D\us{1})=V$, while the second one implies that
$f|_{U\cap D\us{1}}=(f|_{U\cap D\us{1}})^2$. Hence, 
$f$ fixes all elements of $V$.
Since by the definition of $f$ we have $f=f\circ e$, it follows that
$\ker(e)\subseteq\ker(f)\subseteq\ker(e\circ f)$.
The elements of $U\cap D\us{1}$ are in distinct kernel classes
of $e$, because $e$ fixes all elements of $U$. However, some
elements of $U\cap D\us{1}$ are in the same kernel class of $e\circ f$,
because $(e\circ f)(U\cap D\us{1})=e(V)=V\subsetneq U\cap D\us{1}$.
Hence, 
$\ker(e)\subsetneq\ker(e\circ f)$, and therefore
$U=e(A)\supsetneq (e\circ f)(A)$.
Combining the facts that $\delta|_V\subsetneq \theta|_V$
and $e\circ f$ fixes every element of $V$ we also get that
$\theta|_V\subseteq(e\circ f)(\theta)$, and hence
$(e\circ f)(\theta)\not\subseteq\delta$.
The existence of a unary polynomial $e\circ f$ of $\al{A}$ with these
properties contradicts our initial assumption that $U$ is a
$\langle\delta,\theta\rangle$-minimal set of $\al{A}$.
This contradiction shows that $\check{U}$ is a
$\langle\delta^\frakC,\theta^\frakC\rangle$-minimal set of $\al{C}$, and
$\check{N}$ is
$\langle\delta^\frakC,\theta^\frakC\rangle$-trace of $\al{C}$.

Now we will prove that the induced algebras $\al{A}|_N$ and
$\al{C}|_{\check{N}}$ are weakly isomorphic.
Let $P$ be a $k$-ary polynomial operation of $\al{C}$
such that $P(\check{N}^k)\subseteq \check{N}$.
If we write $P$ in the form
described in Corollary~\ref{cor-poly}, where the coordinate polynomials
of $P$ are the $mk$-ary polynomial operations $p\us{1},\dots,p\us{m}$
of $\al{A}$, then
we get that for $2\le i\le m$, $p\us{i}$ must be constant with
value $c\us{i}$ for all allowable inputs. Thus, we may assume without loss of
generality that $p\us{2},\dots,p\us{m}$ are these constant polynomials.
For $i=1$, the assumption $P(\check{N}^k)\subseteq \check{N}$ forces that
the first coordinate function of $P$ assigns to any tuple
$(\check{u}\ls{1},\dots,\check{u}\ls{k})\in\check{N}^k$
the element
$p\us{1}(\check{u}\ls{1},\dots,\check{u}\ls{k})\in\check{N}$.
By the fact that the second through $m$-th coordinates of
$\check{u}\ls{1},\dots,\check{u}\ls{k}$ are $c\us{2},\dots,c\us{m}$,
we have the equality
$p\us{1}(\check{u}\ls{1},\dots,\check{u}\ls{k})=p(u\ls{1},\dots,u\ls{k})$
for all $\check{u}\ls{1},\dots,\check{u}\ls{k}\in\check{N}$
if we denote by $p$ the
$k$-ary polynomial of $\al{A}$
obtained from $p\us{1}$ by
replacing appropriate variables by $c\us{2},\dots,c\us{m}$.
This shows that if $P$ is a $k$-ary polynomial operation of $\al{C}$
such that $P(\check{N}^k)\subseteq \check{N}$, then $\al{A}$ has 
a $k$-ary polynomial operation $p$ such that
$p(N^k)\subseteq N$ and
\begin{equation*}
P|_{\check{N}}(\check{u}\ls{1},\dots,\check{u}\ls{k})=
P|_{\check{N}}\left(
\begin{bmatrix} u\ls{1}\\ c\us{2}\\ \vdots\\ c\us{m}\end{bmatrix},\dots,
\begin{bmatrix} u\ls{k}\\ c\us{2}\\ \vdots\\ c\us{m}\end{bmatrix}
 \right)
 =\begin{bmatrix} p|_N(u\ls{1},\dots,u\ls{k})\\
                    c\us{2}\\ \vdots\\ c\us{m}.
 \end{bmatrix}
\end{equation*}
or equivalently,
\begin{equation}
  \label{eq-weakiso}
  P|_{\check{N}}(\check{u}\ls{1},\dots,\check{u}\ls{k})
  =\bigl(p|_N(u\ls{1},\dots,u\ls{k})\bigr)\check{\phantom{l}}
  \qquad
  \text{for all $u\ls{1},\dots u\ls{k}\in N$}.
\end{equation}  
Conversely.
if $p$ is a $k$-ary polynomial operation of $\al{A}$ satisfying
$p(N^k)\subseteq N$, then the polynomial operation $P$ of $\al{C}$
whose coordinate functions are $p$ and the unary constant polynomials
with values $c\us{2},\dots,c\us{m}$ satisfies these equalities.
Clearly, each one of the induced operations $p|_N$ and $P|_{\check{N}}$
uniquely determines the another.
This proves that if we make a correspondence between the operations
of the non-indexed algebras $\al{A}|_N$ and $\al{C}|_{\check{N}}$
via the assignment $p|_N\mapsto P|_{\check{N}}$ indicated in 
\eqref{eq-weakiso}, then the bijection
$N\to\check{N}$, $u\mapsto\check{u}$ is an isomorphism.
Hence, this bijection is a weak isomorphism $\al{A}|_N\to\al{C}|_{\check{N}}$.
Moreover, by Corollary~\ref{cor-congr}(1), it maps
$\delta|_N$ to $\delta^\frakC|_{\check{N}}$, which
completes the proof of the first statement of Theorem~\ref{thm-tct}.

The second statement now follows easily from the facts that
$\typ(\delta,\theta)$ is the type of the minimal algebra
$\al{A}|_N/\delta|_N$,
$\typ(\delta^\frakC,\theta^\frakC)$ is the type of the minimal algebra
$\al{C}|_{\check{N}}/\delta^\frakC|_{\check{N}}$, we have a weak isomorphism
$\al{A}|_N/\delta|_N\to\al{C}|_{\check{N}}/\delta^\frakC|_{\check{N}}$
induced by the weak isomorphism $\al{A}|_N\to\al{C}|_{\check{N}}$
mapping $\delta|_N$ to $\delta^\frakC|_{\check{N}}$ obtained above, and
weakly isomorphic minimal algebras have the same type; moreover, if that type
is ${\bf2}$, then they also have the same prime characteristic.
\end{proof}

\section{Supernilpotent congruences}
\label{sec-appl1}

A congruence $\alpha$ of an algebra $\al{A}$ is
called \emph{$k$-supernilpotent} if
\[
[\,\underbrace{\alpha,\dots,\alpha}_{\text{$k+1$ $\alpha$'s}}\,]=0,
\]
and $\alpha$ is called \emph{supernilpotent} if it is $k$-supernilpotent
for some $k\ge1$. An algebra $\al{A}$ is called $k$-supernilpotent
or supernilpotent if its congruence $1$ has the property.

It is well known (see e.g.\ \cite[p.\ 370]{AM:SAHC})
and easy to check that for every $k\ge1$ the 4-element
algebra $(\mathbb{Z}_4;+,2x_1\dots x_k)$ ($k\ge2$) is $k$-supernilpotent,
but not $(k-1)$-supernilpotent. Therefore, even for finite algebras of
a fixed size, there is no a priori bound on the 
arity of the higher commutator $[1,\dots,1]$ to be checked
if one wants to determine whether the algebra is
supernilpotent. Hence, it is not clear from the definition whether
supernilpotence is a decidable property for
congruences of finite algebras.

Under mild assumptions on a finite algebra in a finite language,
a combination of basic facts from tame congruence theory
(see \cite{hobby-mckenzie}) and results from \cite{Ke:CMVS} and \cite{AM:SAHC}
yields the following characterization of supernilpotence, which implies that
supernilpotence for these algebras is decidable.

\begin{thm}
  \label{thm-sn-alg}
  \cite{hobby-mckenzie, Ke:CMVS, AM:SAHC}
Let $\al{A}$ be a finite algebra in a finite language such that
the variety $\var{V}(\al{A})$ generated by $\al{A}$ omits type {\bf 1}.
Then $\al{A}$ is supernilpotent if and only if $\al{A}$ factors as a direct
product of nilpotent algebras of prime power order.
\end{thm}  

In more detail, the three theorems on finite algebras $\al{A}$, 
which combine to yield
the characterization in Theorem~\ref{thm-sn-alg},
are as follows:
\begin{enumerate}
\item[(I)]
  If $\var{V}(\al{A})$ omits type {\bf 1}
  and $\al{A}$ is supernilpotent or nilpotent,
  then $\al{A}$ is solvable, and hence
  $\var{V}(\al{A})$ is congruence
  permutable (i.e., $\al{A}$ has a Maltsev term).
  \cite[Thm.~7.2, Cor.~7.6, Thm.~7.11]{hobby-mckenzie}
\item[(II)]
  If $\al{A}$ is nilpotent in a finite language and
  $\var{V}(\al{A})$ is congruence modular, then
  \begin{itemize}
  \item
    $\al{A}$ factors as a direct product of nilpotent algebras of
    prime power order if and only if
    $\al{A}$ has a finite bound on the arities of nontrivial commutator terms. 
    \cite[Thm.~3.14]{Ke:CMVS}
  \end{itemize}
\item[(III)]
  Assuming $\var{V}(\al{A})$ is congruence permutable,
  \begin{itemize}
  \item
    if $\al{A}$ is supernilpotent, then $\al{A}$ is nilpotent; moreover,
  \item
    $\al{A}$ is supernilpotent if and only if $\al{A}$ has a finite bound on the
    arities of nontrivial commutator terms.
    \cite[Cor.~6.15, Lm.~7.5]{AM:SAHC}
  \end{itemize}
\end{enumerate}  

Our goal in this section is to use
the techniques developed in Sections~2--3 to lift
the characterization of supernilpotence in Theorem~\ref{thm-sn-alg}
from algebras to congruences
as follows.

\begin{thm}
  \label{thm-sn}
  Let $\al{A}$ be a finite algebra in a finite language such that
  $\var{V}(\al{A})$ omits type {\bf 1}.
  For any congruence $\alpha$ of $\al{A}$ the following conditions
  are equivalent:
\begin{enumerate}
\item[{\rm(a)}]
  $\alpha$ is supernilpotent.
\item[{\rm(b)}]
  either $\alpha=0$, or else $\alpha$ is nilpotent, and $\al{A}$ has
  congruences $\beta_1,\dots,\beta_\ell\le\alpha$ (for some $\ell>0$) such that 
\begin{enumerate}
\item[{\rm(1)}]
  $\beta_1\wedge\dots\wedge\beta_\ell = 0$,
\item[{\rm(2)}]
 $(\beta_1\wedge\dots\wedge\beta_{i-1})\circ\beta_i = \alpha$
  for every $i\in [\ell]$, $i>1$, and  
\item[{\rm(3)}]
  for each $i\in[\ell]$
  there exists a prime $p_i$ such that every block of $\alpha/\beta_i$ in
  $\al{A}/\beta_i$ has size a power of $p_i$.
\end{enumerate}
\end{enumerate}
\end{thm}

\begin{proof}
  Suppose $\al{A}$ satisfies the assumptions of the theorem,
  let $\lngF$ denote the (finite) language of $\al{A}$, and let
  $\var{V}:=\var{V}(\al{A})$.
  Let us consider any congruence $\alpha$ of $\al{A}$.
  The statement of the theorem is trivial if
  $\alpha=0$, therefore we will assume from now on that $\alpha>0$.
  Let us choose and fix an algebra $\al{I}=([m];\lngF)$ isomorphic to
  the (finite) algebra $\al{A}/\alpha$, and let us fix an onto homomorphism
  $\chi\colon \al{A}\to\al{I}$.
  Thus, $(\al{A},\chi)$ is an object of the category $\Vcomma$.

  Now let $\al{C}:=\frak{C}(\al{A},\chi)$.
  Clearly, $\al{C}$ is a finite algebra in a finite language,
  $\hat{\lngF}_{\al{I}}$, which belongs to the variety $\var{V}^\frakC$
  consisting of all isomorphic copies of algebras in the category
  $\frak{C}\Vcomma$ (cf.\ Corollary~\ref{cor-var}).
  Clearly, $\var{V}(\al{C})$ is a subvariety of $\var{V}^\frakC$.
  Since $\var{V}$ omits type {\bf 1}, and omitting type {\bf 1}
  is characterized, for locally finite varieties, by an idempotent
  Maltsev condition (see \cite[Thm.~9.6]{hobby-mckenzie}),
  Corollary~\ref{cor-clone-idem}(2) implies that
  $\var{V}(\al{C})$ also satisfies this Maltsev condition, and therefore
  $\var{V}(\al{C})$ omits type {\bf 1}.
  In summary, our discussion in this paragraph shows that $\al{C}$ satisfies
  the assumptions of Theorem~\ref{thm-sn-alg}.

  Let us return to the congruence $\alpha$ of $\al{A}$.
  By Corollary~\ref{cor-congr}(1), its image
  under $\frak{C}$ is the congruence $\alpha^\frakC=1$ of $\al{C}$,
  and by Theorem~\ref{thm-comm}, $\alpha$ is a supernilpotent congruence of
  $\al{A}$ if and only if $1$ is a supernilpotent congruence of $\al{C}$,
  that is, if and only if the algebra $\al{C}$ is supernilpotent.
  By Theorem~\ref{thm-sn-alg}, the latter condition holds if and only if
  $\al{C}$ factors as a direct product of nilpotent algebras of
  prime power order; that is, $\al{C}$ is nilpotent and $\al{C}$
  factors as a direct product of algebras of prime power order.
Using congruences this condition can be expressed as follows:
  \begin{enumerate}
  \item[(b)$^*$]
    the congruence $1$ of $\al{C}$ is nilpotent, and $\al{C}$ has congruences
    $\gamma_1,\dots,\gamma_\ell$ (for some $\ell>0$) such that 
\begin{enumerate}
\item[{\rm(1)}]
  $\gamma_1\wedge\dots\wedge\gamma_\ell = 0$,
\item[{\rm(2)}] \label{it:direct}
$(\gamma_1\wedge\dots\wedge\gamma_{i-1})\circ\gamma_i = 1$
for every $i\in [\ell]$, $i>1$,
and  
\item[{\rm(3)}]
  for each $i\in[\ell]$
  there exists a prime $p_i$ such that the algebra $\al{C}/\gamma_i$
  has size a power of $p_i$.
\end{enumerate}
\end{enumerate}
  Condition (2) here is equivalent to requiring that
    the natural homomorphism
    $\nu_i\colon\al{C}\to(\al{C}/\gamma_1)\times\dots\times(\al{C}/\gamma_i)$
    with kernel $\gamma_1\wedge\dots\wedge\gamma_i$ is surjective for every
    $i\in [\ell]$, $i>1$. Hence, (1)--(2) hold iff $\nu_\ell$
    yields an isomorphism 
    $\al{C}\cong(\al{C}/\gamma_1)\times\dots\times(\al{C}/\gamma_\ell)$.

By Theorem~\ref{thm-comm},
$1=\alpha^\frakC$ is a nilpotent congruence of $\al{C}$
  if and only if $\alpha$ is a nilpotent congruence of $\al{A}$.
  Furthermore, using the bijection ${}^\frakC$ from Corollary~\ref{cor-congr}
  between the interval $I(0,\alpha)$ of the congruence lattice of $\al{A}$
  and the congruence lattice of $\al{C}$,
  which also preserves $\wedge$ and $\circ$,
  we can write each $\gamma_i$ as $\gamma_i=\beta_i^\frakC$ for a unique
  $\beta_i\in I(0,\alpha)$, and we see that the existence of
  $\gamma_1,\dots,\gamma_\ell$ with properties (1)--(3) is equivalent to the
  existence of $\beta_1,\dots,\beta_\ell\in I(0,\alpha)$
  such that conditions (1)--(3) in (b) hold. For translating condition (3) in
  (b)$^*$ to condition (3) in (b) we also use the fact that the
  underlying set of each $\al{C}/\beta_i^\frakC$ is the product of
  the blocks of the congruence $\alpha/\beta_i$ of $\al{A}/\beta_i$.
\end{proof}

\section{The Subpower Membership Problem}
\label{sec-appl2}

Let $\class{K}$ be a finite set of finite algebras in a finite language.
The \emph{Subpower Membership Problem} for $\class{K}$ is the following
combinatorial decision problem:

\begin{enumerate}
\item[]
  $\SMP(\class{K})$
\item[]
  Input: $a_1,\dots,a_k,b\in\al{A}_1\times\dots\times\al{A}_n$ with
  $\al{A_1},\dots,\al{A}_n\in\class{K}$.
\item[]
  Question: Is $b$ a member of the subalgebra of
  $\al{A}_1\times\dots\times\al{A}_n$ generated
  by the set $\{a_1,\dots,a_k\}$?
\end{enumerate}
For background and recent results on the subpower membership problem, see
\cite{Ma:SMP, bul-mayr-steidl, steindl1, steindl2, shriner, BMS:SMP}.

In this section we 
will assume that the set $\class{K}$ of algebras fixed
for the Subpower Membership Problem satisfies the following condition
for some integer $\ddd\ge2$:
\begin{multline}
  \label{eq-gen-assumption}
  \qquad
  \text{$\var{V}$ is a variety in a finite language $\lngF$
    with a $\ddd$-cube term, and}\\
  \text{$\class{K}$ is a finite set of finite algebras in $\var{V}$}.
  \qquad
\end{multline}
For the definition of a $\ddd$-cube term, the reader is referred to
\cite{bimmvw, parallelogram}.
Note also that 
every variety with a cube term is congruence modular
by \cite[Thm.~4.2]{bimmvw} (for an easy proof, see
\cite{dent-kearnes-szendrei}),
therefore we may use concepts and results from the theory
of the modular commutator (see \cite{freese-mckenzie}).

It was proved in \cite[Thm.~6.4]{BMS:SMP}
that under assumption \eqref{eq-gen-assumption} on $\class{K}$,
$\SMP(\class{K})\in\PP$ provided
$\class{K}$ generates a residually small variety. This was done by reducing
the general problem $\SMP(\class{K})$ to a `well-structured' subproblem.

In this section we will employ the techniques developed in
Sections~\ref{sec-constr}--\ref{sec-functorC} to further reduce
this subproblem of $\SMP(\class{K})$, and use this reduction to extend
the result of \cite[Thm.~6.4]{BMS:SMP} 
on $\SMP(\class{K})\in\PP$ mentioned above
to a wider family of sets $\class{K}$.

Our starting point for the new reduction will be the reduction proved
in \cite{BMS:SMP}, therefore
we need to recall the relevant concepts and results from \cite{BMS:SMP}.

\begin{df}[{\cite[Def.~6.2]{BMS:SMP}}]
\label{def-dcoh}
Let $a_1,\dots,a_k,b\in\al A_1\times\dots\times\al A_n$
($\al A_1,\dots,\al A_n\in\class{K}$) be an input for $\SMP(\class{K})$
where $a_r=(a_{r1},\dots,a_{rn})$ ($r\in[k]$) and $b=(b_1,\dots,b_n)$.
We call this input \emph{$\ddd$-coherent} if the following
conditions are satisfied:
\begin{enumerate}
\item[(i)]
$n\ge\max\{\ddd,3\}$;
\item[(ii)]
$\al A_1,\dots,\al A_n$ are similar subdirectly irreducible algebras, 
  and each $\al A_\ell$ has abelian monolith $\mu_\ell$; let $\rho_\ell$
  denote the centralizer of $\mu_\ell$;
\item[(iii)]
  for all $I\subseteq[n]$ with $|I|<\max\{\ddd,3\}$,
  the subalgebra of $\prod_{i\in I}\al A_i$
  generated by $\{a_1|_I,\dots,a_k|_I\}$ is a subdirect subalgebra of
  $\prod_{i\in I}\al A_i$, and $b|_I$ 
  is a member of this subalgebra;
\item[(iv)]
for all $i,j\in[n]$,
the subalgebra of $\al A_i/\rho_i\times\al A_j/\rho_j$ generated by
\[
\{(a_{1i}/\rho_i,a_{1j}/\rho_j),\dots,(a_{ki}/\rho_i,a_{kj}/\rho_j)\}
\]
is the graph of an isomorphism $\al A_i/\rho_i\to\al A_j/\rho_j$.
\end{enumerate}
\end{df}

It is easy to see that $\ddd$-coherence for inputs
of $\SMP(\class{K})$ can be checked in polynomial time.

\begin{df}[{\cite[Def.~6.3]{BMS:SMP}}]
\label{def-dcohSMP}
We define $\SMPd(\class{K})$ to be 
the restriction of $\SMP(\class{K})$ to $\ddd$-coherent inputs.
\end{df}

\begin{thm}[{\cite[Thm.~6.4]{BMS:SMP}}]
\label{thm-polyequiv}
If $\var{V}$ is a variety in a finite language with a $\ddd$-cube term,
then the decision problems $\SMP(\class{K})$ and
$\SMPd(\HH\SSS\class{K})$ are polynomial time equivalent for
every finite family $\class{K}$ of finite algebras in $\var{V}$.
\end{thm}

Now we are ready to define our new reduction for $\SMP(\class{K})$.

\begin{df}
\label{def-dcentr}
An input $a_1,\dots,a_k,b\in\al A_1\times\dots\times\al A_n$
($\al A_1,\dots,\al A_n\in\class{K}$) for $\SMP(\class{K})$ will be
called \emph{$\ddd$-central} if it satisfies  
conditions (i) and (iii) in Definition~\ref{def-dcoh} and the following
new condition:
\begin{enumerate}
\item[(ii)$'$]
$\al A_1,\dots,\al A_n$ are subdirectly irreducible algebras such that 
  the monolith $\mu_\ell$ of each $\al A_\ell$ is
  a central congruence of $\al{A}_\ell$  (i.e., $[1,\mu_\ell]=0$,
  hence in particular, $\mu_\ell$ is abelian), and
  the monoliths $\mu_1,\dots,\mu_n$ have the same prime characteristic.
\end{enumerate}
\end{df}  

As before, it is clear that $\ddd$-centrality for inputs
of $\SMP(\class{K})$ can be checked in polynomial time.

\begin{df}
\label{def-dcentrSMP}
We define $\SMPdc(\class{K})$ to be 
the restriction of $\SMP(\class{K})$ to $\ddd$-central inputs.
\end{df}

Our reduction theorem will reduce the solution of $\SMP(\class{K})$
to the solution of $\SMPdc(\class{K}^\star)$ for several new sets
$\class{K}^\star$
of algebras constructed from $\class{K}$, which are defined as follows.

\begin{df}
\label{def-newsets}
Given $\class{K}$ as in \eqref{eq-gen-assumption},
let $\bar{\bar{\class{K}}}$ denote the set
of all subdirectly irreducible algebras $\al{S}$ in 
$\HH\SSS\class{K}$
whose monolith $\mu_{\al{S}}$ is abelian.
Recall that similarity of subdirectly irreducible algebras is an
equivalence relation on $\bar{\bar{\class{K}}}$,
so let
$\bar{\class{K}}_1,\dots,\bar{\class{K}}_q$ denote its equivalence classes.
For each $\ell\in[q]$, let $\al{I}_\ell=([m_\ell],\mathcal{F})$ be a fixed
algebra that is isomorphic to $\al{S}/(0:\mu_{\al{S}})$ for all
$\al{S}\in\bar{\class{K}}_\ell$, and let
\[
\class{K}^\star_\ell:=
\{\frak{C}(\al{S},\chi):\al{S}\in\bar{\class{K}}_\ell,
\ \text{$\chi$ is an onto homomorphism $\al{S}\to\al{I}_\ell$
  with kernel $(0:\mu_{\al{S}})$}\},
\]
where $\frak{C}$ is the functor with domain category
$\bigl(\Alg(\lngF)\,{\downdownarrows}\,\al{I}_\ell\bigr)$.
\end{df}

\begin{thm}
  \label{thm-dcentr-reduc}
  Let $\var{V}$ be a variety in a finite language $\lngF$
  with a $\ddd$-cube term.
  For any finite set $\class{K}$ of finite algebras in $\var{V}$,
  the sets $\class{K}^\star_1,\dots,\class{K}^\star_q$
  of algebras have the following properties:
  \begin{enumerate}
   \item[{\rm(1)}]
    Each $\class{K}^\star_\ell$ $(\ell\in[q])$ is a finite set of finite
    algebras in a variety in a finite language with a $\ddd$-cube term.
   \item[{\rm(2)}]
    Each $\class{K}^\star_\ell$ $(\ell\in[q])$
    is a set of subdirectly irreducible algebras whose monoliths
    are central and have the same prime characteristic.
  \item[{\rm(3)}]
    $\SMP(\class{K})$ is polynomial time reducible to the problems
    $\SMPdc(\class{K}^\star_\ell)$ \mbox{$(\ell\in[q])$}, and the sets
    $\class{K}^\star_\ell$ $(\ell\in[q])$ of algebras
    can be computed from $\class{K}$ in constant time.
  \end{enumerate}
\end{thm}  

\begin{proof}
  We will use all notation introduced in Definition~\ref{def-newsets}.
  For the proofs of statements (1) and (2) we fix an $\ell\in[q]$.
  
  To prove (1), 
  let $\var{W}_\ell$ denote the subvariety of $\var{V}$ generated by
  $\bar{\class{K}}_\ell$. Since
  \begin{equation}
    \label{eq-pairs}
\{(\al{S},\chi):\al{S}\in\bar{\class{K}}_\ell,
\ \text{$\chi$ is an onto homomorphism $\al{S}\to\al{I}_\ell$
  with kernel $(0:\mu_{\al{S}})$}\}
\end{equation}
is a finite set of objects in $\Wcommaell$ with all
$\al{S}\in\bar{\class{K}}_\ell$ finite, it follows from 
  Corollary~\ref{cor-var} that
  $\class{K}^\star_\ell$ is a finite set of finite algebras in
  the variety $\var{W}_\ell^*$ (whose language
  $\hat{\mathcal{F}}_{\al{I}_\ell}$ is finite).
  Since $\var{V}$ has a $\ddd$-cube term, so does its subvariety $\var{W}_\ell$.
  The existence of a $\ddd$-cube term is a strong idempotent Maltsev condition,
  therefore the variety $\var{W}_\ell^*$ also has a
  $\ddd$-cube term by Corollary~\ref{cor-clone-idem}.

  For statement (2), let $\al{C}:=\frak{C}(\al{S},\chi)$ be any algebra
  in $\class{K}^\star_\ell$. Here $(\al{S},\chi)$ is a member of the set in
  \eqref{eq-pairs}, so
  $\al{S}$ is subdirectly irreducible
  with abelian monolith $\mu_{\al{S}}$, and for the centralizer
  $\alpha_{\al{S}}=(0:\mu_{\al{S}})\,(\ge\mu_{\al{S}})$ we have that
  $\alpha_{\al{S}}=\ker(\chi)$.
  The map $^*$ described in Corollary~\ref{cor-congr} is an isomorphism
  between the interval $I(0,\alpha_{\al{S}})$ in the congruence lattice of
  $\al{S}$ and the congruence lattice of $\al{C}$. Therefore,
  $\al{C}$ is subdirectly irreducible with monolith $\mu_{\al{S}}^*$.
  Moreover, by Corollary~\ref{cor-centr}, we have that
  $(0:\mu_{\al{S}}^*)=(0:\mu_{\al{S}})^*=\alpha^*=1$, so
  the monolith $\mu_{\al{S}}^*$ of $\al{C}$ is central. 
  Since $\al{S}\in\var{V}$ and $\var{V}$ is congruence modular,
  Theorem~\ref{thm-tct} and the fact that $\mu_{\al{S}}$ is
  abelian imply that
  $\typ_{\al{C}}(0,\mu_{\al{S}}^*)=\typ_{\al{S}}(0,\mu_{\al{S}})={\bf2}$, and
  the congruences $\mu_{\al{S}}^*$ and $\mu_{\al{S}}$ --- i.e., 
  the prime quotients $\langle 0,\mu_{\al{S}}^*\rangle$ and
  $\langle 0,\mu_{\al{S}}\rangle$ --- have the same prime characteristic.
  Since all algebras $\al{S}$ in \eqref{eq-pairs} are similar, and hence
  by the definition or by the characterization
  of similarity in \cite[Def.~10.6, Thm.~10.8]{freese-mckenzie}
  they have the same characteristic,
  we conclude that the monoliths of all members of
  $\class{K}^\star_\ell$ have the same prime characteristic.
  
  In statement (3) it is clear that the sets
  $\class{K}^\star_\ell$ $(\ell\in[q])$ of algebras
  can be computed from $\class{K}$ in constant time.
  We need to argue that $\SMP(\class{K})$ is polynomial
  time reducible to the problems
  $\SMPdc(\class{K}^\star_\ell)$ ($\ell\in[q]$).
  In view of Theorem~\ref{thm-polyequiv}, it suffices to show that
  $\SMPd(\HH\SSS\class{K})$ is polynomial time reducible to the problems
  $\SMPdc(\class{K}^\star_\ell)$ ($\ell\in[q]$). We will prove this by
  describing how to assign to every input $a_1,\dots,a_k,b$ for
  $\SMPd(\HH\SSS\class{K})$
  a number $\ell\in[q]$ and an input
  $\tilde{a}_1,\dots,\tilde{a}_k,\tilde{b}$ for
  $\SMPdc(\class{K}^\star_\ell)$ in such a way that
  \begin{itemize}
  \item
    the answer to the input $a_1,\dots,a_k,b$ is `yes'
    if and only if the answer
    to the input $\tilde{a}_1,\dots,\tilde{a}_k,\tilde{b}$ is `yes', and
  \item
    $\tilde{a}_1,\dots,\tilde{a}_k,\tilde{b}$ can be computed from
    $a_1,\dots,a_k,b$ in polynomial time.
  \end{itemize}  
  
Let $a_1,\dots,a_k,b\in\prod_{j\in[n]}\al A_j$
($\al A_1,\dots,\al A_n\in\HH\SSS\class{K}$) be an input for
$\SMPd(\HH\SSS\class{K})$.
Then this is a $\ddd$-coherent input for $\SMP(\class{K})$, that is,
conditions (i)--(iv) from Definition~\ref{def-dcoh} hold.
In particular, (ii) implies that there is a unique
$\ell\in[q]$ such that $\al{A}_1,\dots,\al{A}_n\in\bar{\class{K}}_\ell$.
This $\ell$ can be found in constant time.
Now let us fix $\al{I}:=\al{I}_\ell$ as described in
Definition~\ref{def-newsets},
choose an isomorphism $\phi_1\colon\al{A}_1/\rho_1\to\al{I}$, and
for $2\le j\le n$ define the isomorphisms
$\phi_j\colon\al{A}_j/\rho_j\to\al{I}$
by $\phi_j:=\phi_1\circ\iota_{j,1}$ where $\iota_{j,1}$ is the isomorphism
$\al{A}_j/\rho_j\to\al{A}_1/\rho_1$ from condition~(iv) in
Definition~\ref{def-dcoh}.
For each $j\in[n]$ let $\chi_j\colon\al{A}_j\to\al{I}$ be defined by
$\chi_j:=\phi_j\circ\nu_j$ where $\nu_j\colon\al{A}_j\to\al{A}_j/\rho_j$
is the natural homomorphism.
So, $\ker(\chi_j)=\rho_j=(0:\mu_j)$ for every $j\in[n]$.
Clearly, the algebra $\al{I}$ and each isomorphism $\iota_{j,1}$ can
be computed in constant time, so the homomorphisms $\chi_j$ ($j\in[n]$)
can be computed in $O(n)$ time.

Let $\al{B}$ and $\al{B}^+$ denote the subalgebras of $\prod_{j\in[n]}\al{A}_j$
generated by the sets $\{a_1,\dots,a_k\}$ and $\{a_1,\dots,a_k,b\}$,
respectively.
Using the homomorphisms $\chi_j$ $(j\in[n])$ we can express
condition (iv) as follows:
$\al{B}$ is a subalgebra of the algebra $\prodI_{j\in[n]}\al{A}_j$, where
$(\prodI_{j\in[n]}\al{A}_j;\chi)$ with $\chi$ defined by
$(x_1,\dots,x_n)\mapsto\chi_1(x_1)=\dots=\chi_n(x_n)$, is
the product of the objects $(\al{A}_j,\chi_j)$ ($j\in[n]$)
in the category $\comma$.
Condition (iii) of Definition~\ref{def-dcoh} implies that
$\al{B}$ and $\al{B}^+$ are both subdirect subalgebras of
$\prod_{j\in[n]}\al{A}_j$, moreover,
the requirements on $b$ in this condition ensure that 
$\al{B}^+$ is a subalgebra of $\prodI_{j\in[n]}\al{A}_j$ as well.
Therefore,
the map $\chi$ restricts to $\al{B}$ and $\al{B}^+$ as onto
homomorphisms $\chi|_{\al{B}}\colon\al{B}\to\al{I}$
and $\chi|_{\al{B}^+}\colon\al{B}^+\to\al{I}$.
The kernel classes of these homomorphisms are the sets
$B\cap D\us{i}$ and $B^+\cap D\us{i}$
($i\in[m]$), respectively, where $D\us{i}=\prod_{j\in[n]}\chi^{-1}\ls{j}(i)$.
Hence,
$B^+\cap D\us{i}\supseteq B\cap D\us{i}\not=\emptyset$
for every $i\in[m]$.

We can compute
representatives $d\us{i}\in B\cap D\us{i}$ for each $i\in[m]$ by
generating all elements of 
$\al{B}/\ker(\chi|_{\al{B}})\cong\al{I}$ using $\chi(a_1),\dots,\chi(a_k)$
(at most $m$ distinct elements),
and replicating the same computation using at most $m$ of the generators
$a_1,\dots,a_k$ instead, one from each set
$\{a_1,\dots,a_k\}\cap\chi^{-1}\bigl(\chi(a_r)\bigr)$ ($r\in[k]$).
This requires $O(kn)$ time. 

Now let $\al{C}_j:=\frak{C}(\al{A}_j,\chi_j)$ for every $j\in[n]$, and
let $\al{B}^*$ and $(\al{B}^+)^*$ 
be the subalgebras of $\prod_{j\in[n]}\al{C}_j$
obtained from $\al{B}$ and $\al{B}^+$ by applying the map ${}^*$
described in Theorem~\ref{thm-subprod}(1).
Furthermore, let
$\tilde{\phantom{G}}\colon \prodI_{j\in[n]}A_j\to \prod_{j\in[n]} C_j$
be the function
defined in Theorem~\ref{thm-subprod}(3), using the elements
$d\us{i}$ ($i\in[m]$) from the preceding paragraph as
padding elements.
Let $\tilde{a}_1,\dots,\tilde{a}_k,\tilde{b}$ be the elements of
$(B^+)^*$ obtained from the given input elements $a_1,\dots,a_k,b$ 
by applying this function. Clearly,
$\tilde{a}_1,\dots,\tilde{a}_k,\tilde{b}$ can be computed from
$a_1,\dots,a_k,b$ and $d\us{1},\dots,d\us{m}$ in $O\bigl(kn\bigr)$ time.

Since $\{a_1,\dots,a_k\}$ generates $\al{B}$, and
$\{a_1,\dots,a_k,b\}$ generates $\al{B}^+$, 
Theorem~\ref{thm-subprod}(3), together with the fact 
$d\us{1},\dots,d\us{m}\in B\subseteq B^+$, implies 
that
\begin{equation}
  \label{eq-generates}
  \begin{matrix}
  \text{$\{\tilde{a}_1,\dots,\tilde{a}_k\}$ is a generating set for $\al{B}^*$,
      and}\hfill\\
\text{$\{\tilde{a}_1,\dots,\tilde{a}_k,\tilde{b}\}$ is a generating set for
  $(\al{B}^+)^*$.}\hfill
\end{matrix}
\end{equation}
  
Now we want to show that
\begin{equation}
  \label{eq-new-input}
\tilde{a}_1,\dots,\tilde{a}_k,\tilde{b}\in\al{C}_1\times\dots\times\al{C}_n
\ (\al{C}_1,\dots,\al{C}_n\in\class{K}^\star_\ell)
\end{equation}
is a correct input for
$\SMPdc(\class{K}^\star_\ell)$, that is, $\tilde{a}_1,\dots,\tilde{a}_k,\tilde{b}$
is a $\ddd$-central input for $\SMP(\class{K}^\star_\ell)$.
We need to check that conditions (i), (ii)$'$, and (iii) from
Definition~\ref{def-dcentr} hold for the new input \eqref{eq-new-input},
that is, they hold for
$\tilde{a}_r$ ($r\in[k]$), $\tilde{b}$,
$\al{C}_j$ ($j\in[n]$), and $\class{K}^\star_\ell$ in place of
$a_r$ ($r\in[k]$), $b$,
$\al{A}_j$ ($j\in[n]$), and $\class{K}$.
Note first that our general assumption \eqref{eq-gen-assumption}
(suppressed in Definition~\ref{def-dcentr}) holds for $\class{K}^\star_\ell$
by Theorem~\ref{thm-dcentr-reduc}(1).
Condition (i) from Definition~\ref{def-dcentr}
holds for the new input \eqref{eq-new-input},
because it is identical to condition (i) of
$\ddd$-coherence for the original input $a_1,\dots,a_k,b$.
Condition (ii)$'$ from Definition~\ref{def-dcentr}
holds for \eqref{eq-new-input}
by Theorem~\ref{thm-dcentr-reduc}(2), since
$\al{C}_1,\dots,\al{C}_n\in\class{K}^\star_\ell$.
To check condition (iii) from Definition~\ref{def-dcentr}
for the new input \eqref{eq-new-input},
let $I\subseteq[n]$ be such that $|I|<\max\{\ddd,3\}$.
In view of \eqref{eq-generates},
the subalgebra of $\prod_{j\in I}\al{C}_j$ generated by
$\{\tilde{a}_1|_I,\dots,\tilde{a}_k|_I\}$ is $\al{B}^*|_I$, while
the subalgebra of $\prod_{j\in I}\al{C}_j$ generated by
$\{\tilde{a}_1|_I,\dots,\tilde{a}_k|_I,\tilde{b}|_I\}$
is $(\al{B}^+)^*|_I$.
Therefore, condition (iii) from Definition~\ref{def-dcentr}
for \eqref{eq-new-input} and for the chosen $I$
is equivalent to saying
that $\al{B}^*|_I$ projects to each coordinate $j\in I$ to be $\al{C}_j$,
and $(\al{B}^+)^*|_I=\al{B}^*|_I$.
To prove that $\al{B}^*|_I$ and $(\al{B}^+)^*|_I$ meet these conditions,
notice that our assumption that
the original input $a_1,\dots,a_k,b$ is $\ddd$-coherent, implies that
the analogous conditions hold for $\al{B}|_I$, $\al{B}^+|_I$, and $\al{A}_j$
$(j\in I)$.
Consequently, they also hold if we pass to their images under
the map ${}^*$
described in Theorem~\ref{thm-subprod}(1).
Thus,
$(\al{B}|_I)^*$ projects to each coordinate $j\in I$ to be $\al{C}_j$,
and $(\al{B}^+|_I)^*=(\al{B}|_I)^*$.
According to Theorem~\ref{thm-subprod}(2),
projection onto a set of coordinates is preserved by $^*$, 
therefore these conditions hold with the order of $|_I$ and $^*$ switched.
This proves (iii) from Definition~\ref{def-dcentr}
for the new input \eqref{eq-new-input}, and hence
finishes the proof that the input \eqref{eq-new-input} is
$\ddd$-central.

Finally, using again \eqref{eq-generates}, we get that
\begin{align*}
\text{the an}&\text{swer of $\SMPdc(\class{K}^\star_\ell)$ to the input
  $\tilde{a}_1,\dots,\tilde{a}_k,\tilde{b}$ is `yes'}\\
& \Leftrightarrow\quad
\tilde{b}\in\al{B}^*
\quad\Leftrightarrow\quad
(\al{B}^+)^*=\al{B}^*\quad
\stackrel{\text{Thm~\ref{thm-subprod}}}{\Leftrightarrow}\quad
(\al{B}^+)=\al{B}
\quad\Leftrightarrow\quad
b\in\al{B}\\
& \Leftrightarrow\quad
\text{the answer of $\SMPd(\HH\SSS\class{K})$ to the input
  $a_1,\dots,a_k,b$ is `yes'}.
\end{align*}
The crucial step is $\stackrel{\text{Thm~\ref{thm-subprod}}}{\Leftrightarrow}$,
where we use the bijective property of the map ${}^*$ in
Theorem~\ref{thm-subprod}(1).
The proof of Theorem~\ref{thm-dcentr-reduc} is complete.
\end{proof}

\begin{cor}
  \label{cor-P-reduc}
  The following conditions are equivalent:
  \begin{enumerate}
  \item[{\rm(a)}]
    $\SMP(\class{K})\in\PP$ for every finite set $\class{K}$ of finite
    algebras in a variety in a finite language with a $\ddd$-cube term.
  \item[{\rm(b)}]
    $\SMPdc(\class{K})\in\PP$ for every finite set $\class{K}$ of finite
    algebras in a variety in a finite language with a $\ddd$-cube term.
  \end{enumerate}
\end{cor}

\begin{thm}
  \label{thm-P-snilp-centrs}
  Let $\var{V}$ be a variety in a finite language
  $\lngF$ with a $\ddd$-cube term.
  If $\class{K}$ is a finite set of finite algebras in $\var{V}$ such that
  \begin{enumerate}
  \item[$(\ddagger)$]
    for every subdirectly irreducible algebra $\al{S}\in\HH\SSS\class{K}$
    with abelian monolith $\mu_{\al{S}}$ the centralizer of $\mu_{\al{S}}$
    is supernilpotent,
  \end{enumerate}
  then $\SMP(\class{K})\in\PP$.
\end{thm}  

The result of \cite[Thm.~6.4]{BMS:SMP} mentioned earlier is a special case
of Theorem~\ref{thm-P-snilp-centrs}; it is obtained from this theorem
by replacing the word `supernilpotent' in condition $(\ddagger)$
by the word `abelian'. Indeed,
the strengthening of condition $(\ddagger)$ where `supernilpotent' is
replaced by `abelian' is equivalent to the condition
that $\class{K}$ generates a residually small variety, and hence 
$\var{V}$ can be chosen to be residually small
(see. e.g., \cite[Cor.~2.4]{BMS:SMP} and the discussion preceding it).

Finite algebras with a Maltsev term that satisfy a slightly weaker 
condition than $(\ddagger)$, namely that in every subdirectly
irreducible homomorphic image the monolith has supernilpotent centralizer, 
already appeared in~\cite{Ma:MASC}
where a relational description of their polynomial clones was given.

\begin{proof}[Proof of Theorem~\ref{thm-P-snilp-centrs}]
  We will use the notation introduced in Definition~\ref{def-newsets}.
  By Theorem~\ref{thm-dcentr-reduc}(3),
  it suffices to show that under the assumptions
  of Theorem~\ref{thm-P-snilp-centrs} we have 
  $\SMPdc(\class{K}^\star_\ell)\in\PP$ for every $\ell\in[q]$.
  Let $\ell\in[q]$ be fixed for the rest of the proof.
  For notational convenience, we will drop the subscript $\ell$; that is,
  we will write
  $\bar{\class{K}}$ for $\bar{\class{K}}_\ell$,
  $\al{I}=([m];\mathcal{F})$ for $\al{I}_\ell=([m_\ell];\mathcal{F})$, and
  $\class{K}^\star$ for $\class{K}^\star_\ell$.
  Furthermore, let $\al{K}:=\prod\class{K}^\star$ be the product of all members of
  $\class{K}^\star$.
  Clearly, the variety $\var{V}(\al{K})$ it generates coincides with the variety
  generated by $\class{K}^\star$.

  We know from Theorem~\ref{thm-dcentr-reduc}(1) that
  $\class{K}^\star$ is a finite set of finite
  algebras in a variety in a finite language (namely, $\hat{\lngF}_{\al{I}}$)
  with a $\ddd$-cube term.
  Therefore, $\al{K}$ is a finite algebra, the variety $\var{V}(\al{K})$
  has a $\ddd$-cube term, and hence is congruence modular.
    
  In the proof of Theorem~\ref{thm-dcentr-reduc}(2) we saw
  that for every algebra $\al{C}=\frak{C}(\al{S},\chi)$ in
  $\class{K}^\star$,
  \begin{itemize}
  \item
    $\al{C}$ is subdirectly irreducible with monolith $\mu_{\al{S}}^*$,
    where $\mu_{\al{S}}$ is the monolith of $\al{S}$, and
  \item
    $1=\alpha_{\al{S}}^*$ centralizes $\mu_{\al{S}}^*$, where
    $\alpha_{\al{S}}=(0:\mu_{\al{S}})$.
  \end{itemize}  
  Now, our additional assumption $(\ddagger)$ implies that
  $\alpha_{\al{S}}$
  is a supernilpotent congruence of $\al{S}$.
  Hence, by Theorem~\ref{thm-comm},
  $1=\alpha_{\al{S}}^*$ is a supernilpotent
  congruence of $\al{C}$.
  Thus, all algebras in $\class{K}^\star$ are supernilpotent.
  Since they are also subdirectly irreducible, 
  Theorem~\ref{thm-sn-alg} implies that they are nilpotent algebras
  of prime power order.
  By Theorem~\ref{thm-dcentr-reduc}(2) we also have that the algebras in
  $\class{K}^\star$ have the same prime characteristic.
  It follows that the cardinality of every algebra in $\class{K}^\star$
  is a power of the same prime.
  Hence, the product of these algebras, $\al{K}$, is also nilpotent of
  prime power order.
  Statement~(I) following Theorem~\ref{thm-sn-alg} also implies that
  $\al{K}$ has a Maltsev term.

  Now we can use \cite[Thm.~1.2]{Ma:SMP} to conclude that
  $\SMP(\{\al{K}\})\in\PP$. Clearly, $\SMP(\{\al{K}\})$ is polynomial time
  reducible to $\SMP(\class{K}^\star)$, since every computation for
  $\SMP(\{\al{K}\})$ can be viewed as a computation for $\SMP(\class{K}^\star)$.
  Conversely, $\SMP(\class{K}^\star)$ is also polynomial time reducible to
  $\SMP(\{\al{K}\})$ for the following reason: by \cite[Thm.~4.11]{BMS:SMP},
  if $\class{L}$ is a finite set of finite algebras in a variety
  in a finite language with a cube term, then
  $\SMP(\class{L})$ and $\SMP(\HH\SSS\class{L})$ are polynomial time
  equivalent. This theorem applies to
  $\class{L}=\{\al{K}\}$, so since 
  $\class{K}^\star\subseteq\HH\SSS\{\al{K}\}$, we get that
  $\SMP(\class{K}^\star)$ is a subproblem of $\SMP(\HH\SSS\{\al{K}\})$,
  which is polynomial time equivalent to $\SMP(\{\al{K}\})$.
  This proves that $\SMP(\class{K}^\star)\in\PP$.
  It follows that for its subproblem we also have
  $\SMPdc(\class{K}^\star)\in\PP$, which completes the proof of
  Theorem~\ref{thm-P-snilp-centrs}.
\end{proof}

\bibliographystyle{plain}

\end{document}